\documentclass[a4paper, 11pt]{amsart}
\usepackage{amsmath, amsthm}
\usepackage{hyperref}
\usepackage{mathbbol}
\usepackage{graphicx}

\newtheorem{thm}{Theorem}
\newtheorem{prop}[thm]{Proposition}
\newtheorem{lemma}[thm]{Lemma}
\theoremstyle{definition}
\newtheorem{defi}[thm]{Definition}
\newtheorem{cor}[thm]{Corollary}
\theoremstyle{remark}

\newtheorem{rem}[thm]{Remark}

\newcommand{\alg}[1]{\mathfrak{{#1}}}
\newcommand{\tr}{\text{tr}}
\newcommand{\str}{\text{str}}

\newcommand{\rn}[1]{\mathbb{R}^{{#1}}}
\newcommand{\co}[2]{\left[{#1},{#2}\right]} 
\newcommand{\aco}[2]{\left\{{#1},{#2}\right\}}
\newcommand{\eref}[1]{(\ref{#1})} 

\newcommand{\pderi}[2]{ { \frac{\partial {#1} }{\partial {#2} } } }
\newcommand{\pd}[2]{ { \frac{\partial {#1} }{\partial {#2} } } }

\newcommand{\sh}[1]{\mathcal{#1}}
\newcommand{\C}{{\mathbb{C}}}
\newcommand{\R}{{\mathbb{R}}}
\newcommand{\Z}{{\mathbb{Z}}}
\DeclareMathOperator{\Hom}{Hom}
\DeclareMathOperator{\End}{End}
\newcommand{\WA}{{\mathcal{A}}}
\newcommand{\gf}{\C}
\newcommand{\tW}{ {[[u]]\otimes_{\gf[u]} W} }  
\newcommand{\p}{\partial}
\newcommand{\pl}{\overleftarrow{\p}}
\newcommand{\pr}{\overrightarrow{\p}}
\newcommand{\pdl}[2]{ { \frac{\pl {#1} }{\p {#2} } } }

\normalsize

\begin{document}
\title{Cyclic Cohomology of the Weyl Algebra}

\author{Thomas Willwacher}
\address{Department of Mathematics, ETH Zurich}
\email{thomas.willwacher@math.ethz.ch}
\thanks{The author was partially supported by the Swiss National Science Foundation (grant 200020-105450)}
\date{}

\maketitle

\begin{abstract}
We give an explicit formula for $\alg{sp}_{2n}$-basic representatives of the cyclic cohomology of the Weyl algebra $HC^{\bullet}(\WA_{2n})$. This paper can be seen as cyclic addendum to the paper \cite{felder} by Feigin, Felder and Shoikhet, where the analogous Hochschild case was treated.

As an application, we prove a generalization of a Theorem of Nest and Tsygan concerning the relation of the Todd class and the cyclic cohomology of the differential operators on a complex manifold.
\end{abstract}


\section{Introduction and Notations}
The polynomial Weyl algebra $\WA_{2n}$ is the algebra of polynomial differential operators on $\rn{n}$. It is isomorphic to the algebra of polynomials on $\rn{2n}$ with product the Moyal product associated to the standard symplectic structure
\[
\omega = \sum_{j=1}^n dp_j\wedge dq_j = \sum_{j=1}^n dy_{2j-1}\wedge dy_{2j}.
\]
Here $y_1=p_1,y_2=q_2, \dots ,y_{2n-1}=p_{n},y_{2n}=q_{n}$ are the standard coordinates on $\rn{2n}$.

Note that the symplectic Lie algebra $\alg{sp}_{2n}$ is a Lie subalgebra of $\WA_{2n}$, namely the Lie subalgebra of homogeneous quadratic polynomials.

We will be interested in the Hochschild and cyclic cohomology of $\WA_{2n}$. The definition of the Hochschild cochain complex $C^\bullet(\WA_{2n},\WA_{2n}^*)$ is standard and recited in the appendix. On the other hand, there are several distinct versions of cyclic cohomology. However, there is a unified treatment due to Getzler. The cyclic cochain complex is given by
\[
CC_W^\bullet(\WA_{2n}) = C^\bullet(\WA_{2n},\WA_{2n}^*)\tW
\]
where $W$ is some $\gf[u]$-module with $u$ being a formal variable of degree 2.
The differential is given by $d+uB$, where $d$ is the Hochschild differential and $B$ is the Connes differential, both of which are defined in the appendix. As the module $W$ varies, one gets several distinct cyclic cohomology theories. For example, in the special case where $W=\gf$ is the trivial $\gf[u]$-module (i.e., $u$ acts as 0), one recovers the usual Hochschild complex.

In the cases of interest in this paper, the Hochschild and cyclic cohomology of the Weyl algebra are well known:
\begin{prop}
The Hochschild cohomology of $\WA_{2n}$ satisfies
\begin{align*}
HH^{j}(\WA_{2n},\WA_{2n}^*) &\cong
\begin{cases}
\C & \text{for $j=2n$}  \\
0 \ & \text{otherwise}
\end{cases}
\end{align*}
If $W$ is $\gf[u]$ or $\gf(u)$, then
\begin{align*}
HC^{\bullet}(\WA_{2n}) &\cong HH^{\bullet}(\WA_{2n},\WA_{2n}^*)\tW.
\end{align*}
\end{prop}

It is also not too difficult to write down explicit representatives for the above cohomology classes in the normalized Hochschild or cyclic complexes, using the fact that $\WA_{2n}=(\WA_2)^{\otimes n}$. However, for applications in deformation quantization and index theory we would like this representative, say $\tau$, to have the additional property that it is $\alg{sp}_{2n}$-basic, namely
\begin{equation}
\label{equ:iotatauis0}
\iota_a \tau = 0
\end{equation}
for all $a\in \alg{sp}_{2n}\subset \WA_{2n}$, and with the interior product $\iota_a$ as defined in Appendix \ref{app:morestruct}. This condition is important when globalizing local constructions using formal geometry. The details will be clarified later, in particular in Sections \ref{sec:fedreminder} and \ref{sec:cmplxformal}. For now, just note that it implies that $\tau$ is $\alg{sp}_{2n}$-invariant, since $a\in \alg{sp}_{2n}$ acts on $\tau$ as
\[
L_a \tau = \aco{d+uB}{\iota_a}\tau =\left( (d+uB)\iota_a+\iota_a (d+uB) \right)\tau = 0+0 =0.
\]
Remarkably, it is quite difficult to write down representatives $\tau$ for the cohomology that satisfy \eref{equ:iotatauis0}. For the Hochschild case ($W=\C$) a solution has been found by Feigin, Felder and Shoikhet (=:FFS) using Shoikhet's proof of Tsygan's formality conjecture. In this paper, we will give representatives satisfying \eref{equ:iotatauis0} for the cyclic case. The first main result is:

\begin{thm}
\label{thm:tauw}
The cyclic cochain
\[
\tau_w = e^{-u\iota_\omega}\tau_{2n}\otimes w
\]
for an arbitrary $w\in W\setminus \{0\}$ is a non-trivial cyclic cocycle and satisfies $\iota_a \tau_w=0$ for all $a \in \alg{sp}_{2n}$.
\end{thm}

The notation will be explained in the next sections. In particular, $\tau_{2n}$ is the cocycle found by FFS and the operator $\iota_\omega$ is defined in \eref{equ:iotaomegadef}.

Furthermore, as described in \cite{felder} and Section \ref{sec:lacohomcc}, there is natural map $ev_1$ from the Hochschild cohomology of an algebra $A$ to the Lie algebra cohomology of the Lie algebra $A$. FFS have shown that their cocycle is mapped to the $2n$-component of a characteristic class defined by the product of the $A$-roof genus and the Chern character. Our second main result is that the whole characteristic class (all components) is the image of the cocycle $\tau_w$.

\begin{thm}
\label{thm:tau2kisach}
The image of $\tau_{2k}^r=\frac{1}{(n-k)!}(-\iota_\omega)^{n-k}\tau_{2n}^r$ under $ev_1$ defines the same relative Lie algebra cohomology class as the $2k$ component of $\hat{A}Ch$, up to a sign, i.e.
\[
[ev_1(\tau_{2k}^r)] = (-1)^k [\hat{A}Ch]_{2k} \in H^{\bullet}(\alg{g},\alg{h})
\]
for all $k=0,..,n$.
\end{thm}

The notation used here will be introduced in Section \ref{sec:lacohomcc}.

\begin{rem}
The cyclic cocycle and the above results have been obtained independently by Pflaum, Posthuma and Tang and can be found in their upcoming paper \cite{ppt}. Furthermore, they consider more and deeper applications than are considered here. To the interested reader, the author strongly recommends the lecture of their work.
\end{rem}

\subsection{Structure of the paper}
In the next section, we review the formula given by FFS for the representative $\tau_{2n}$ and show that it satisfies \eref{equ:iotatauis0}. In section \ref{sec:cycformula} we derive from this the representative in the cyclic case and prove the first main result, Theorem \ref{thm:tauw}.

In Section \ref{sec:charclasses}, we partially repeat the treatment of Lie algebra cohomology and characteristic classes given by FFS, and prove the second main result, Theorem \ref{thm:tau2kisach}.

In Section \ref{sec:applications} we consider several applications, mostly generalizations of the applications already discussed by FFS. In particular, Proposition \ref{prop:tauistdch} in section \ref{sec:tdch} generalizes a Theorem by Nest and Tsygan on the relation of the Todd class and the cyclic cohomology of differential operators on a complex manifold (Theorem 7.1.1 in \cite{NT}).

Some standard definitions and sign conventions have been compiled in Appendix \ref{app:stddef}, in order not to interrupt the main line of arguments.

\subsection{Acknowledgements}
The author wants to thank his advisor Prof. Giovanni Felder for his continuous support and scientific guidance.  Furthermore he is grateful to X. Tang for sharing a preliminary version of his work on the subject.

\section{The FFS formula}
\label{sec:FFSformula}
In this section, we recall some of the results obtained by Feigin, Felder and Shoikhet \cite{felder} in the Hochschild case.
\subsection{Some Definitions}
\label{sec:somedefs}
To simplify the formulas, we make the following definitions:
On a tensor product of functions $a_0\otimes \cdots \otimes a_k$ we define the operator ${}_i\partial_\alpha$ to be the partial derivative wrt. $y_\alpha$ applied to the $i$-th factor, i.e.,
\[
{}_i\partial_\alpha (a_0\otimes \cdots \otimes a_k) 
=
a_0\otimes \cdots \otimes \pderi{a_i}{y_\alpha} \otimes \cdots \otimes a_k
\]

Furthermore, we define
\[
\alpha^i_{rs} = ({}_r\partial_{2i-1})({}_s\partial_{2i}) - ({}_r\partial_{2i})({}_s\partial_{2i-1}) 
\]
and\footnote{Note the missing factors of 1/2 in comparison to the definition of \cite{felder}.}
\[
\alpha_{rs} = \sum_{i=1}^{n} \alpha^i_{rs}.
\]

Next, define the operator $\mu_k$ acting on a tensor product of $(k+1)$ functions as 
\begin{equation}
\label{equ:mudef}
\mu_{k}(a_0\otimes \cdots \otimes a_k) = a_0(0)a_1(0)\cdots a_k(0).
\end{equation}
I.e., it is just the evaluation at 
$y_1=y_2=\cdots =y_{2n}=0$
followed by multiplication.

\subsection{The Formula}
\label{sec:FFStheformula}
The Hochschild $2n$-cocycle is the composition of three maps: 
\[
\tau_{2n} = \mu_{2n} \circ S_{2n} \circ \pi_{2n}.
\]
The leftmost map has already been defined above. The rightmost map is given by  
\begin{equation}
\label{equ:pi2ndef}
\pi_{2n} = \det
\begin{pmatrix}
{}_1\partial_1 & {}_1\partial_2 &  \cdots & {}_1\partial_{2n} \\
\vdots & \vdots & \ddots & \vdots \\
{}_{2n}\partial_{1} & {}_{2n}\partial_{2} &  \cdots & {}_{2n}\partial_{2n}.
\end{pmatrix}
\end{equation}
Finally, the middle one is given by 
\begin{equation}
\label{equ:Sdef}
S_{2n} = \int_{\Delta_{2n}} du_1\cdots du_{2n} \prod_{0\leq i< j \leq 2n} e^{B_1(u_j-u_i)\alpha_{ji}} 
\end{equation}
where $B_1(u)=u-\frac{1}{2}$ is the first Bernoulli polynomial and the integral is taken over the $2n$-simplex $\{ 0=u_0< u_1 <\cdots< u_{2n}< 1\}$.

We will take the following theorem for granted.
\begin{thm}[FFS \cite{felder}]
The cochain $\tau_{2n}$ is a cocycle, i.e., $d\tau_{2n}=0$.
\end{thm}

To show that $\iota_a \tau_{2n}=0$ for all $a \in \alg{sp}_{2n}$, FFS use the following lemma:
\begin{lemma}[Lemma 2.2 of \cite{felder}]
\label{lem:22} 
Let $\sigma\in \Sigma_{2n}$ be a permutation. Then 
\begin{align*}
\tau_{2n}(a_0\otimes a_{\sigma^{-1}(1)}\otimes\cdots\otimes a_{\sigma^{-1}(2n)})) \\
&= 
\text{sgn}(\sigma)\tau_{2n}^\sigma(a_0\otimes\cdots\otimes a_{2n})  
\end{align*}
where 
\[
\tau_{2n}^\sigma = \mu_{2n}\circ \int_{\sigma(\Delta_{2n})} du_1\cdots du_{2n} \prod_{0\leq i< j \leq 2n} e^{b_1(u_j-u_i)\alpha_{ji}} \circ \pi_{2n}.
\]
Here $b_1(x)$ is the 1-periodic extension of the Bernoulli polynomial $B_1(x)$ from the unit interval $[0,1)$ to $\R$. The domain of integration is $\sigma(\Delta_{2n})=\{0=u_0<u_{\sigma^{-1}(1)}<\cdots <u_{\sigma^{-1}(2n)}<1\}$.
\end{lemma}

Using this lemma, we next repeat the proof by FFS of the fact that $\iota_a \tau_{2n}=0$. Let $\sigma_j \in \Sigma_{2n}$ be the permutation 
\[
1\ 2\ 3\ ..\ 2n \mapsto 2\ 3\ ..\ j\ 1\ (j+1)\ ..\ 2n.
\]
Then one can write, using the previous lemma
\begin{align*}
& (\iota_a \tau_{2n})(a_0,..,a_{2n-1})=\sum_{j=0}^{2n-1} \tau_{2n}^{\sigma_j}(a_0,a,a_1,..,a_{2n-1}) \\
&=\mu_{2n}\circ \sum_{j=0}^{2n}\int_{\sigma_j(\Delta_{2n})} du_1\cdots du_{2n} \prod_{0\leq i< j \leq 2n} e^{b_1(u_j-u_i)\alpha_{ji}} \circ \pi_{2n} (a_0,a,a_1,..,a_{2n}) \\
&=\mu_{2n}\circ \int_0^1 du_1 \int_{\Delta_{2n-1}} du_2\cdots du_{2n}
\prod_{0\leq i< j \leq 2n} e^{b_1(u_j-u_i)\alpha_{ji}} \circ \pi_{2n} (a_0,a,a_1,..,a_{2n})
\end{align*}
Note that since $a\in \alg{sp}_{2n}$ is quadratic, it needs to receive exactly two derivatives, otherwise the resulting contribution in the above formula is 0. The map $\pi_{2n}$ contributes one derivative, so the other derivative needs to come from one of the $\alpha_{j1}$. Hence the coefficient will contain exactly one term of the form $b_1(u_1-u_i)$ containing $u_1$. But this term vanishes upon performing the $u_1$-integral by the elementary fact that $\int_0^1(x-1/2)dx=0$. Hence it is shown that $\iota_a \tau_{2n}=0$.

\section{The cyclic case}
\label{sec:cycformula}
\subsection{Another differential}
Define the operator $\iota_\omega$ on $C^\bullet(\WA_{2n},\WA_{2n}^*)$ to be
\begin{equation}
\label{equ:iotaomegadef}
\iota_\omega := \frac{1}{2} \sum_{i,j} \omega_{ij} \iota_{y_i} \iota_{y_j}=\sum_j \iota_{p_j} \iota_{q_j}.
\end{equation}
Also define $L_\omega := \co{d}{\iota_\omega}$.
\begin{lemma}
\label{lem:someLids}
The following holds:
\begin{align*}
\aco{d}{L_\omega} &=0  
&\co{L_\omega}{\iota_\omega}&=0 \\
\aco{L_\omega}{L_\omega} &=2L_\omega^2=0 
&\co{\iota_\omega}{B}&=0 \\
\aco{L_\omega}{B}&=0 
\end{align*}
In particular, $(B-L_\omega)$ is a differential of degree $-1$ on $C^\bullet(\WA_{2n},\WA_{2n}^*)$.
\end{lemma}

\begin{proof}
The first equality is an immediate consequence of the definition. The second equality immediately follows from the second equality of Lemma \ref{lem:morestruct} in the appendix and the fact that $\iota_1\equiv 0$ on the normalized chain complexes. The third equality is a direct consequence of the first two. The last two equalities follow from the last two equalities of Lemma \ref{lem:morestruct}. The fact that $(B-L_\omega)$ is a differential is clear from the third and fifth equalities.
\end{proof}

Technically the main result of this section is that $\tau_{2n}$ is closed w.r.t. this differential.
\begin{thm}
\label{thm:DBtauis0}
$(B-L_\omega)\tau_{2n}=0$
\end{thm}

As a direct corollary, we can construct the desired cocycles representing cyclic cohomology classes, and prove the first main result, Theorem \ref{thm:tauw}.

\begin{proof}[Proof of Theorem \ref{thm:tauw}]

Since $\co{\iota_\omega}{\iota_a}=0$, the second property ($\iota_a \tau_w=0$) follows immediately from the analogous statement for $\tau_{2n}$.

To prove the first part, i.e., that $\tau_w$ is a cyclic cocycle, compute:
\begin{align*}
d\tau_w
&=
e^{-u\iota_\omega} \left(e^{u\iota_\omega}  d e^{-u\iota_\omega}\right) \tau_{2n}\otimes w \\
&=
e^{-u\iota_\omega} \left(d+u\co{\iota_\omega}{d}+\frac{u^2}{2}\co{\iota_\omega}{\co{\iota_\omega}{d}} + \dots\right)\tau_{2n} \otimes w \\
&=
-e^{-u\iota_\omega} uL_\omega\tau_{2n}\otimes w
=
-ue^{-u\iota_\omega}B\tau_{2n}\otimes w = -uB\tau_w.
\end{align*}
Hence $(d+uB)\tau_w=0$. For the third equality, we used that $d\tau_{2n}=0$ and the fact that $\co{L_\omega}{\iota_\omega}=0$ from Lemma \ref{lem:someLids}. 

It remains to show that $\tau_w$ is not trivial, i.e., not a coboundary. If it were a coboundary, it would evaluate to $0$ on the cycle $c=1\otimes p_1\wedge q_1\wedge \dots p_n\wedge q_n$. However, $\tau_w(c) = \tau_{2n}(c)\otimes w=(2n)!\neq 0$, where the middle equality was already observed by FFS.
\end{proof}

\subsection{The proof of Theorem \ref{thm:DBtauis0}}

Theorem \ref{thm:DBtauis0} will be proven by the following two lemmas, each of which computes one of the two terms occuring in its statement.

\begin{lemma}
The following holds:
\begin{multline*}
(B\tau_{2n}) (a_1\otimes .. \otimes a_{2n}) = \\
= \mu_{2n}\circ \int_{0=u_1<u_2<..<u_{2n}<1} du_2\cdots du_{2n} \prod_{1\leq i< j \leq 2n} e^{b_1(u_j-u_i)\alpha_{ji}} \circ \pi_{2n} (1\otimes a_1\otimes .. \otimes a_{2n})
\end{multline*}
\end{lemma}
\begin{proof}
Let $C$ be the cyclic permutation of $2n$ elements. Then 
\[
(B\tau_{2n}) (a_1\otimes .. \otimes a_{2n}) = \sum_{j=0}^{2n-1} (-1)^j \tau_{2n}(1\otimes C^j(a_1\otimes .. \otimes a_{2n})) = \sum_{j=0}^{2n-1} \tau_{2n}^{C^j}(1\otimes a_1\otimes .. \otimes a_{2n})
\]
by Lemma \ref{lem:22}. Furthermore,
\begin{align*}
\sum_{j=0}^{2n-1} \tau_{2n}^{C^j} 
&=
\mu_{2n}\circ 
\left( \sum_{j=0}^{2n-1} \int_{C^j(\Delta_{2n})} \right)
du_1\cdots du_{2n} \prod_{0\leq i< j \leq 2n} e^{b_1(u_j-u_i)\alpha_{ji}} \circ \pi_{2n} \\
&=
\mu_{2n}\circ 
\int_{\Gamma}
du_1\cdots du_{2n} \prod_{0\leq i< j \leq 2n} e^{b_1(u_j-u_i)\alpha_{ji}} \circ \pi_{2n}
\end{align*}
where the integration domain $\Gamma$ is the configuration space of $2n$ ordered points on the circle. In our case, terms in the product for which $i=0$ will not contribute, since any derivative of 1 vanishes. Hence the integrand is invariant w.r.t. joint shifts of the $u_1$,..,$u_{2n}$. We can use this by changing the integration variables to
\[
u_j' =
\begin{cases}
u_1 & \quad \text{for $j=1$} \\
u_j-u_1 & \quad \text{otherwise}
\end{cases}
\]
and integrating out $u_1'$. The result is that 
\begin{multline*}
(B\tau_{2n}) (a_1\otimes .. \otimes a_{2n}) = \\
= \mu_{2n}\circ \int_{0<u_2'<..<u_{2n}'<1} du_2'\cdots du_{2n}' \prod_{1\leq i< j \leq 2n} e^{b_1(u_j'-u_i')\alpha_{ji}} \circ \pi_{2n} (1\otimes a_1\otimes .. \otimes a_{2n})
\end{multline*}
where we put $u_1':=0$ to simplify notation. This was to be proved.
\end{proof}

\begin{lemma}
The following holds:
\begin{multline*}
L_\omega\tau_{2n} (a_1\otimes .. \otimes a_{2n}) = \\
= \mu_{2n}\circ \int_{0=u_1<u_2<..<u_{2n}<1} du_2\cdots du_{2n} \prod_{1\leq i< j \leq 2n} e^{b_1(u_j-u_i)\alpha_{ji}} \circ \pi_{2n} (1\otimes a_1\otimes .. \otimes a_{2n})
\end{multline*}
\end{lemma}
\begin{proof}
Note that 
\[
L_\omega=\co{d}{\iota_\omega}
=\frac{1}{2} \sum_{i,j} \omega_{ij} (L_{y_i}\iota_{y_j} -\iota_{y_i} L_{y_j})
=\sum_{i,j} \omega_{ij} L_{y_i}\iota_{y_j}.
\]
As above (see section \ref{sec:FFStheformula}) we obtain that
\begin{multline*}
(\iota_{y_k}\tau_{2n})(a_0\otimes .. \otimes a_{2n-1}) =
- \mu_{2n}\circ \int_0^1 du_{2n}  \int_{0=u_0<u_1<..<u_{2n-1}<1} du_2\cdots du_{2n-1}
\\
\prod_{0\leq i< j \leq 2n} e^{b_1(u_j-u_i)\alpha_{ji}} \circ \pi_{2n} (a_0\otimes .. \otimes a_{2n-1}\otimes y_k).
\end{multline*}
Since $y_k$ is linear and there is already one derivative acting on it through $\pi_{2n}$, all terms contributing nontrivially are independent of $u_{2n}$. Hence the $u_{2n}$-integral can be performed yielding a factor 1. 

Next note that for a $(2n-1)$-cochain $\phi$
\[
L_{y_l}\phi = \phi \circ \left( \sum_{p=0}^{2n-1} \sum_m \omega^{lm} \: {}_p\partial_m \right).
\]
Here $\omega^{lm}$ is the inverse of $\omega_{ij}$, i.e., $\sum_j\omega_{ij}\omega^{jk}=\delta_i^k$. When applying this formula to the case at hand, we obtain that 
\begin{multline*}
L_\omega\tau_{2n} (a_0\otimes .. \otimes a_{2n-1})= \\
=
-\mu_{2n} \int_{0=u_0<u_1<..<u_{2n-1}<1} du_2\cdots du_{2n-1} \prod_{0\leq i< j \leq 2n} e^{b_1(u_j-u_i)\alpha_{ji}} \\
\sum_{k,l} \omega_{kl}
\left( \sum_{p=0}^{2n-1} \sum_m \omega^{lm} \: {}_p\partial_m \right) 
\pi_{2n}
(a_0\otimes .. \otimes a_{2n-1}\otimes y_k).
\end{multline*}
The term on the right is
\begin{align*}
& \sum_{k,l} \omega_{kl}
\left( \sum_{p=0}^{2n-1} \sum_m \omega^{lm} \: {}_p\partial_m \right) 
\pi_{2n}
(a_0\otimes .. \otimes a_{2n-1}\otimes y_k) \\
&= 
\sum_{k} 
\sum_{p=0}^{2n-1} {}_p\partial_k
\pi_{2n}
(a_0\otimes .. \otimes a_{2n-1}\otimes y_k)
\\ &= 
\sum_{p=0}^{2n-1}
\sum_{k}
{}_p\partial_k
\det
\begin{pmatrix}
{}_1\partial_1 & {}_1\partial_2 &  \cdots & {}_1\partial_{2n} \\
\vdots & \vdots & \ddots & \vdots \\
{}_{2n-1}\partial_{1} & {}_{2n-1}\partial_{2} &  \cdots & {}_{2n-1}\partial_{2n} \\
{}_{2n}\partial_1 & {}_{2n}\partial_2 &  \cdots & {}_{2n}\partial_{2n}
\end{pmatrix}
(a_0\otimes .. \otimes a_{2n-1}\otimes y_k)
\\ &= 
\sum_{p=0}^{2n-1}
\sum_{k}
{}_p\partial_k
\det
\begin{pmatrix}
{}_1\partial_1 & {}_1\partial_2 &  \cdots & {}_1\partial_{2n} \\
\vdots & \vdots & \ddots & \vdots \\
{}_{2n-1}\partial_{1} & {}_{2n-1}\partial_{2} &  \cdots & {}_{2n-1}\partial_{2n} \\
\delta^k_1 & \delta^k_2 &  \cdots & \delta^k_{2n}
\end{pmatrix}
(a_0\otimes .. \otimes a_{2n-1}\otimes 1)
\\ &= 
\sum_{p=0}^{2n-1}
\det
\begin{pmatrix}
{}_1\partial_1 & {}_1\partial_2 &  \cdots & {}_1\partial_{2n} \\
\vdots & \vdots & \ddots & \vdots \\
{}_{2n-1}\partial_{1} & {}_{2n-1}\partial_{2} &  \cdots & {}_{2n-1}\partial_{2n} \\
{}_p\partial_1 & {}_p\partial_2 &  \cdots & {}_p\partial_{2n}
\end{pmatrix}
(a_0\otimes .. \otimes a_{2n-1}\otimes 1)
\\ &=
-\det
\begin{pmatrix}
{}_0\partial_1 & {}_0\partial_2 &  \cdots & {}_0\partial_{2n} \\
{}_1\partial_1 & {}_1\partial_2 &  \cdots & {}_1\partial_{2n} \\
\vdots & \vdots & \ddots & \vdots \\
{}_{2n-1}\partial_{1} & {}_{2n-1}\partial_{2} &  \cdots & {}_{2n-1}\partial_{2n}.
\end{pmatrix}
(a_0\otimes .. \otimes a_{2n-1}\otimes 1)
\end{align*}
For the last equality, we used the vanishing of the determinant if $p\neq 0$, since it then contains two equal rows.
It is now easily seen that the lemma follows by renumbering the $a_j$ and the $u_j$. 
\end{proof}

The above two lemmas clearly prove Theorem \ref{thm:DBtauis0} and hence Theorem \ref{thm:tauw}.

\subsection{A slight Generalization}
For later use, we need to know the cyclic cohomology of the algebra $\WA_{2n}^r=\WA_{2n}\otimes \gf^{r\times r}$ of $r$ by $r$ matrices with entries in $\WA_{2n}$. Since this algebra is Morita equivalent to $\WA_{2n}$ the answer is simple, namely the two cohomologies agree. The explicit representatives are given by 
\[
\tau_w^r((a_0\otimes M_0)\otimes \dots \otimes (a_k\otimes M_k)) 
= 
\tau_w(a_0\otimes\dots\otimes a_k) \tr(M_0\cdots M_k).
\]

Note that $\WA_{2n}^r$ contains the Lie subalgebra $\alg{sp}_{2n}\oplus \alg{gl}_r$, where $\alg{gl}_r$ is embedded as the constant matrices, and $\alg{sp}_{2n}$ as the matrices $\alg{sp}_{2n}\otimes \mathbb{1}_{r\times r}$. Note also that $\iota_\alpha \tau_w^r=0$ for all $\alpha \in \alg{sp}_{2n}\oplus \alg{gl}_r \subset \WA_{2n}^r$, since $\iota_1 \tau_w=0$ in the normalized cyclic complex.

\section{Characteristic classes}
\label{sec:charclasses}
Analogously to the treatment given by FFS in section 5 of \cite{felder}, we can consider the relation of $\tau_w^r$ to characteristic Lie algebra cohomology classes. In particular FFS showed that $\tau_{2n}^r$ yields the $2n$-component of the $\hat{A}Ch$-class. Not surprisingly, we can show that $\tau_{2k}^r:=\frac{1}{(n-k)!}(-\iota_\omega)^{n-k}\tau_{2n}^r$ yields the $2k$-component of $\hat{A}Ch$.

To make these statements precise, we first need to repeat some preliminaries from \cite{felder}. This will be done in the next subsection. The precise version of the above claim, i.e. Theorem \ref{thm:tau2kisach}, will be recalled in subsection \ref{sec:cctheclaim}. It is proved in subsection \ref{sec:cctheproof}.

\subsection{Lie algebra cohomology and characteristic classes}
\label{sec:lacohomcc}
For $\alg{g}$ a Lie algebra, $\alg{h}\subset \alg{g}$ a Lie subalgebra and $M$ a $\alg{g}$-module, we denote by $(C^\bullet(\alg{g},\alg{h};M), d)$ the relative Chevalley-Eilenberg complex. Its definition can be found in the appendix. The cohomology of this complex is called the relative Lie algebra cohomology and denoted by $H^\bullet(\alg{g},\alg{h};M)$. If $M$ is omitted, it is understood that $M=\gf$ is the trivial module.

Let $pr:\alg{g}\rightarrow \alg{h}$ be an $\alg{h}$-equivariant projection, i.e., a projection such that 
\[
pr(\co{h}{g}) = \co{h}{pr(g)}
\]
for all $h\in \alg{h}, g\in \alg{g}$. The ``curvature'' $C\in \Hom(\wedge^2 \alg{g}, \alg{h})$ measures the failure of $pr$ to be a Lie algebra homomorphism:
\begin{equation}
\label{equ:curvdef}
C(u\wedge v) = \co{pr(u)}{pr(v)}-pr(\co{u}{v})
\end{equation}

Using $C$, one can construct the Chern-Weil homomorphism
\[
\chi : (S^\bullet\alg{h})^{*\alg{h}} \rightarrow C^{2\bullet}(\alg{g},\alg{h})
\] 
given for some $P\in (S^k\alg{h})^{*\alg{h}}$ as 
\[
\chi(P)(v_1,..,v_{2k}) =  \frac{1}{k!2^k} \sum_{\sigma\in S_{2k}}sgn(\sigma)P(C(v_{\sigma(1)},v_{\sigma(2)}),..,C(v_{\sigma(2k-1)},v_{\sigma(2k)})).
\]
where we identified $P$ with its polarization.

In our case, we take $\alg{g}=\WA_{2n}^r$ and $\alg{h}=\alg{sp}_{2n}\oplus \alg{gl}_r$. The $\alg{h}$-invariant polynomials $P$ we are interested in are the homogeneous components in the power series expansion of 
\[
P(x,M) = \hat{A}(x) Ch(M).
\]
Here $(x,M)\in \alg{sp}_{2n}\oplus \alg{gl}_r$,
\[
\hat{A}(x)= \sideset{}{^{\frac{1}{2}}}\det\left(\frac{x/2}{\sinh(x/2)}\right)
\]
is the $\hat{A}$-genus and 
\[
Ch(M) = \tr(e^M)
\]
is the Chern character. In the definition of $\hat{A}(x)$, the determinant is defined as the composition $\alg{sp}_{2n} \hookrightarrow \alg{gl}_{2n} \stackrel{\det}{\to} \C$.
 
The image of $P$ under $\chi$ defines a characteristic relative Lie algebra cohomology class which we will denote
\[
\hat{A}Ch = [\chi(P)] \in H^{\bullet}(\alg{g},\alg{h}).
\]

Another way to obtain a Lie algebra cohomology class of the Lie algebra $\alg{g}=\WA_{2n}^r$ is by evaluating at $1$ and antisymmetrizing a \emph{Hochschild} cohomology class of the \emph{algebra} $\WA_{2n}^r$. Concretely, there is the following chain map 
\begin{align*}
ev_1 : C^\bullet(\WA_{2n}^r,(\WA_{2n}^r)^*) &\rightarrow H^{\bullet}(\alg{g}) \\
ev_1(\phi)(v_1\wedge\dots\wedge v_k) 
&=
\sum_{\sigma\in S_k} sgn(\sigma) \phi(1\otimes v_{\sigma(1)}\otimes\dots\otimes v_{\sigma(k)}).
\end{align*}

We will use the following convention for embedding wedge products into tensor product space:
\[
v_1\wedge .. \wedge v_{k} := \sum_{\sigma\in S_k} sgn(\sigma) v_{\sigma(1)}\otimes\dots\otimes v_{\sigma(k)}.
\]
Hence, for example $ev_1(\phi)(v_1\wedge\dots\wedge v_k)=\phi(1\otimes v_1\wedge\dots\wedge v_k)$.

Note that, since we are working with the normalized Hochschild complex
\[
ev_1(B\phi) = 0.
\]
The images of the components $\tau_{2k}^r$ of $\tau_w$ will each define a Lie algebra cohomology class since
\[
d\circ ev_1(\tau_{2k}) = ev_1(d\tau_{2k}) = -ev_1(B\tau_{2k+2}) = 0.
\]

\subsection{The claim}
\label{sec:cctheclaim}

Recall that Theorem \ref{thm:tau2kisach}, stated in the introduction, claims that:
\[
[ev_1(\tau_{2k}^r)] = (-1)^k [\hat{A}Ch]_{2k} \in H^{\bullet}(\alg{g},\alg{h})
\]
for all $k=0,..,n$.

For $k=n$ this statement was proved by FFS in \cite{felder}. One can give a lenghthy, but more or less direct proof of the above statement by computing the integrals defining $\tau_{2k}^r$ in a special case. This proof differs slightly from that given by FFS and will be presented in Appendix \ref{app:directproof}. 

Here, however, we will take a much quicker route and reduce the above Theorem to the special case $k=n$, which was already proved by FFS and we take for granted. 

\subsection{The proof of Theorem \ref{thm:tau2kisach}}
\label{sec:cctheproof}
The first step in the reduction is the same as that used by FFS. We copy it here: Let $W_{n,r}\subset \WA_{2n}^r$ be the Lie subalgebra consisting of polynomials of the form $\sum_{j=1}^n f_j(q)p_j\otimes \mathbb{1}_{r\times r} + \sum_{j\geq 1} g_j(q)\otimes M_j$, for $f_j(q),g_j(q)$ polynomials in the $q_1,..,q_n$ and $M_j\in\alg{gl}_r$ matrices. Set further $\alg{h}_1=\alg{h}\cap W_{n,r}$. The following statement was proved by FFS, sections 5.5 and 5.6. 

\begin{lemma}
The map on cohomology
\[
H^{2k}(\WA_{2n}^r, \alg{h}) \rightarrow H^{2k}(W_{n,r}, \alg{h}_1)
\]
induced by restriction of cochains is an injection for $k\leq n$.
\end{lemma}

Hence in order to prove Theorem \ref{thm:tau2kisach}, it will be sufficient to show that the cohomology classes $[ev_1(\tau_{2k}^r)]$ and $[\hat{A}Ch]_{2k}$ are mapped to the same class in $H^{2k}(W_{n,r}, \alg{h}_1)$. Actually, we will show a stronger statement, namely:

\begin{prop}
\label{prop:tauisachonwn}
The restrictions of the cocycles $ev_1(\tau_{2k}^r)$ and $\chi(P_{k})$ to $W_{n,r}$ agree. Here $P_{k}\in (S^k\alg{h})^{*\alg{h}}$ is the degree $k$ component of the invariant polynomial defining the class $\hat{A}Ch$.
\end{prop}
Again, the special case $k=n$ has already been shown by FFS, and we take it for granted. We will begin the proof of the general case by establishing another special case:

\begin{lemma}
\label{lem:tauisachse}
The restrictions of the cocycles $ev_1(\tau_{2k}^r)$ and $\chi(P_{k})$ to $W_{k,r}\subset W_{n,r}$ agree. Here, $W_{k,r}$ consists of those polynomials in $W_{n,r}$ involving only $p_1,q_1,..,p_k,q_k$.
\end{lemma}
\begin{proof}
In this proof, we make the dependence of $\tau_{2k}^r$ on $n$ explicit by writing $\tau_{2k, n}^r=\tau_{2k}^r$. Similarly, we write explicitly $P_{k,n}=P_k$. 
For $v_1,..,v_{2k}\in W_{k,r}$ we compute
\begin{align*}
& \tau_{2k,n}^r(1\otimes v_1\wedge .. \wedge v_{2k}) 
=
\frac{1}{(n-k)!}((-\iota_\omega)^{n-k}\tau_{2n,n})(1\otimes v_1\wedge .. \wedge v_{2k}) \\
&=
(\prod_{j=k+1}^{n}(-\iota_{p_j}\iota_{q_j}) \tau_{2n,n})(1\otimes v_1\wedge .. \wedge v_{2k}) \\
&=
\tau_{2n,n}(1\otimes v_1\wedge .. \wedge v_{2k}\wedge p_{k+1}\wedge q_{k+1} \wedge .. \wedge p_n \wedge q_{n}) \\
&=
\mu_{2n}\circ \int_{0<u_1<u_2<..<u_{2k}<1} du_1\cdots du_{2k} \prod_{1\leq i< j \leq 2k} e^{b_1(u_j-u_i)\alpha_{ji}} \circ \\
& \quad \quad \quad \circ \pi_{2n} (1\otimes (v_1\wedge .. \wedge v_{2k})\otimes p_{k+1}\otimes q_{k+1} \otimes .. \otimes q_{n}) \\
&=
\mu_{2n}\circ S_{2k} \circ \pi_{2k} (1\otimes v_1\wedge .. \wedge v_{2k}) 
= 
\tau_{2k,k}(1\otimes v_1\wedge .. \wedge v_{2k}) \\
&=
\chi(P_{k,k})(v_1\wedge .. \wedge v_{2k})
=
\chi(P_{k,n})(v_1\wedge .. \wedge v_{2k}).
\end{align*}
For the second equality, observe that there are derivatives w.r.t. $p_{k+1},q_{k+1},..,q_n$ present in the definition of $\pi_{2n}$, and since $v_1,..,v_{2k}$ do not contain such $p$'s and $q$'s, they have to be supplied by some $\iota_\omega$. The fourth equality is a routine application of Lemma \ref{lem:22}, similarly to that at the end of section \ref{sec:FFStheformula}. The last but one equality is the special case $k=n$ of Proposition \ref{prop:tauisachonwn}, proved by FFS.
\end{proof}

We proceed to prove Proposition \ref{prop:tauisachonwn}. We have to show that 
\begin{equation}
\label{equ:tauonwn}
\tau_{2k}^r(1\otimes v_1\wedge\dots \wedge v_{2k}) = \chi(P_{k})(v_1\wedge\dots \wedge v_{2k})
\end{equation}
for all $v_1,..,v_{2k}\in W_{n,r}$. 
\begin{lemma}
Suppose all the $v_j$ are homogeneous. Then both sides of \eref{equ:tauonwn} vanish, unless each $v_j$ has one of the following three alternative forms
\[
v_j = 
\begin{cases}
p_{\alpha_j} \otimes \mathbb{1}_{r\times r} & \quad \text{for some $\alpha_j\in\{1,..,n\}$}\\
p_{\alpha_j} q_{\beta_j} q_{\gamma_j} \otimes \mathbb{1}_{r\times r} & \quad \text{for some $\alpha_j,\beta_j,\gamma_j\in\{1,..,n\}$}\\
q_{\alpha_j} \otimes M_j & \quad \text{for some $\alpha_j\in\{1,..,n\}$, $M_j\in \alg{gl}_r$}
\end{cases}
\]
Furthermore, they also vanish unless exactly $k$ of the $v_j$ take on the first form (i.e., $p_{\alpha_j}$).
\end{lemma}
\begin{proof}
For a homogeneous $v_j$, define the total degree as (degree in $p$'s)-(degree in $q$'s). For example, $v_j$ is of total degree $0$ only if $v_j\in\alg{h}_1$. If some $v_j\in\alg{h}_1$, both sides of \ref{equ:tauonwn} vanish by $\alg{h}_1$-basicness. 

The only $v_j$'s with positive total degree are those of the first form, $v_j=p_{\alpha_j}$, which have total degree $1$. Of these, at most $k$ can occur, as can be seen from the definitions of $\pi_{2k}$ and $\chi$.
Note that the remaining $v_j$ must each have total degree at most $-1$, and that on both sides the sum of all total degrees must vanish. Hence one can conclude that all the remaining $v_j$ must have total degree exactly $-1$ and there must be exactly $k$ such terms.
\end{proof}

By the above lemma, and by reordering the factors in the wedge product, we can assume that 
\[
v_j = 
\begin{cases}
p_{\alpha_j} q_{\beta_j} q_{\gamma_j} \otimes \mathbb{1}_{r\times r} & \quad \text{for $j=1,..,m$ }\\
q_{\alpha_j} \otimes M_j & \quad \text{for $j=m+1,..,k$ } \\
p_{\alpha_j} \otimes \mathbb{1}_{r\times r} & \quad \text{for $j=k+1,..,2k$ }
\end{cases}
\]
for some integer $m\in\{0,..,k\}$. 

We next claim that by $GL(n)$-invariance of both sides of \eref{equ:tauonwn}, it is sufficient to consider the case where all $\alpha_j,\beta_j,\gamma_j\in\{1,..,k\}$ for all $j=1,..,2k$. Then we can apply Lemma \ref{lem:tauisachse} to finish the proof. 

First consider the case $m=0$. In this case $v_j = q_{\alpha_j} \otimes M_j$ for $j=1,..,k$ and hence an appropriate permutation in $GL(n)$ can be applied to ensure that $\alpha_j \in \{1,..,k\}$ for all $j=1,..,k$. But then both sides of \eref{equ:tauonwn} vanish by power counting, except if also $\alpha_j \in \{1,..,k\}$ for $j=k+1,..,2k$. Hence both sides of \eref{equ:tauonwn} agree for $m=0$ by Lemma \ref{lem:tauisachse}. 

Next consider the case $m=1$. Then $v_1=p_{\alpha_1} q_{\beta_1} q_{\gamma_1} \otimes \mathbb{1}_{r\times r}$. Choose a $g\in GL(n)$ such that $g^* (q_{\beta_1} q_{\gamma_1}) = q_1^2-c q_2^2$ for some irrelevant constant $c$, being either 0 (in case $\beta_1=\gamma_1$) or 1. Then $g^*v_1$ is a linear combination of terms of the form $p_{\alpha} q_{\beta}^2 \otimes \mathbb{1}_{r\times r}$ for some $\alpha,\beta \in\{1,..,k\}$. By using multilinearity and $GL(n)$-invariance of both sides of \eref{equ:tauonwn}, we can hence assume from the start that $v_1=p_{\alpha_1} q_{\beta_1}^2$. But from here, the same reasoning as in the $m=0$-case can be applied to establish \eref{equ:tauonwn} for $m=1$.

Finally consider the general case ($m=2,3,..$). To begin with, assume that $v_j$ for $1\leq j\leq m$ depends only on $q_1,..,q_{j}$ and the $p$'s, i.e., is independent of $q_{j+1},..,q_{j}$. Let us say that such $v_1,..,v_m$ satisfy the \emph{flag property}. In particular, this implies that $v_1,..,v_m$ depend only on $q_1,..,q_{m}$ and the $p$'s. But then, by renumbering the coordinates appropriately, we can assume that $v_1,..,v_k$ are independent of $q_{k+1},..,q_n$, and again apply the same reasoning as in the $m=0$- case to show \eref{equ:tauonwn}. 

Hence the proof of Proposition \ref{prop:tauisachonwn} will be complete if we can prove the following technical lemma.
\begin{lemma}
Let the monomials $v_1,..,v_m$ have the form $v_j = p_{\alpha_j} q_{\beta_j} q_{\gamma_j}$, $\alpha_j,\beta_j,\gamma_j\in\{1,..,n\}$. Then for $j=1,..,m$, $s=1,2,..$ there are monomials $w_{j,s}$ of the same form together with elements $g_s\in GL(n)$ and constants $c_s$ such that
\[
v_1\wedge ..\wedge v_m = \sum_{s,j} c_s g_s^*(w_{1,s}\wedge .. \wedge w_{m,s})
\]
and such that $w_{1,s},..,w_{m,s}$ have the flag property for all $s$.
\end{lemma}
\begin{proof}
The proof is by induction on $m$. For $m=0$ there is nothing to show. Assume the Lemma is true for $m-1$ and we want to show it for $m$. Then by the induction hypothesis we can assume w.l.o.g. that $v_1,..,v_{m-1}$ have the flag property. Distinguish the following four cases:
\begin{enumerate}
\item If $v_m=p_\alpha q_\beta q_\gamma$ with $\beta,\gamma\leq m$, then $v_1,..,v_m$ already have the flag property and there is nothing to show.
\item If $\beta<m$ and $\gamma > m$, or vice versa, let $g\in GL(n)$ be the permutation of the coordinates $\beta$ and $m$. Then 
\[
v_1\wedge ..\wedge v_m = g^*\left[v_1\wedge ..\wedge (g^*v_m)\right]
\]
and the term in the outer brackets has the flag property.
\item If $\beta=\gamma > m$ the same treatment as in the previous case applies.
\item If $\beta,\gamma > m$ and $\beta\neq \gamma$, let $g\in GL(n)$ be the transformation defined by 
\[
g^*q_i=
\begin{cases}
q_\beta+q_\gamma & \quad \text{for $i=\beta$} \\
q_\beta-q_\gamma & \quad \text{for $i=\gamma$} \\
q_i & \quad \text{otherwise}.
\end{cases}
\]
Then $g^*v_m= (g^*p_\alpha) (q_\beta^2- q_\gamma^2)$ is a linear combination of terms of the form considered in the previous case, and hence the treatment there applies. 
\end{enumerate}

\end{proof}

\section{Applications}
\label{sec:applications}
Here we list several applications of the results obtained in the previous sections. We suppose that the reader is already familiar with Fedosov's deformation quantization procedure for symplectic manifolds.

\subsection{Brief Reminder of the Fedosov construction}
\label{sec:fedreminder}
See \cite{fedosov} for details.
Let $(M, \omega)$ be a compact $n$-dimensional symplectic manifold. Let $\sh{F}=\prod_{j\geq 0}S^j T^*M[[\epsilon]]$ be the bundle of formal fiberwise (f.f.) functions on $M$. It is a bundle of algebras with fiberwise product on the fiber $\sh{F}_x$ (over $x\in M$) defined as the Moyal product wrt. the symplectic form $\epsilon\omega_x$. 

Let $\nabla$ be a symplectic connection on $M$. It extends to a connection on $\sh{F}$, also denoted by $\nabla$. Fedosov showed that there is a certain section $A$ of $\sh{F}\otimes \wedge^1 T^*M$, such that the connection $D=\nabla+\epsilon^{-1}\co{A}{\cdot}$ is flat. Explicitly,
\[
D^2 = \nabla^2+\co{\epsilon^{-1} \nabla A + \frac{\epsilon^{-2}}{2}\co{A}{A} }{\cdot} = \co{\epsilon^{-1}F+\epsilon^{-1}\nabla A + \frac{\epsilon^{-2}}{2}\co{A}{A} }{\cdot}=0
\]
where $F$ is the section of $S^2 T^*M$ such that $\nabla^2=\epsilon^{-1}\co{F}{\cdot}$. Flatness means that the section $\Omega = F + \nabla A + \frac{\epsilon^{-1}}{2}\co{A}{A}$ of $\sh{F}\otimes \wedge^2 T^*M$ is central, which in turn implies that it is a section of $\wedge^2 T^*M[[\epsilon]]$. 

Furthermore, note that the product on the fibers of $\sh{F}$ induces a product on the space of flat sections of $\sh{F}$.

Any smooth function $f \in \WA_\epsilon := C^\infty(M)[[\epsilon]]$ can be interpreted as a section of $S^0 T^*M[[\epsilon]]\subset \sh{F}$. In fact, it can be uniquely lifted to a flat section $\hat{f}$ of $\sh{F}$ (wrt. $D$) and every flat section of $\sh{F}$ arises in this way. In this way the product on the space of flat sections gives rise to a product $\star$ on $\WA_\epsilon$, which is a deformation quantization of the usual commutative product. The 2-form $\Omega$ is called the characteristic class of $\star$.

\subsection{A chain map}
Let 
\[
CC^{per}_\bullet(\WA_\epsilon, \WA_\epsilon) = C_\bullet(\WA_\epsilon, \WA_\epsilon)(u^{-1})
\]
be the periodic cyclic chain complex of $\WA_\epsilon$ with differential $b+u^{-1}B'$ (see the Appendix for the definitions). We can construct an $\gf[[\epsilon]](u^{-1})$-linear chain map 
\[
\Phi : (CC^{per}_\bullet(\WA_\epsilon, \WA_\epsilon), b+u^{-1}B') \rightarrow (\Omega^\bullet(M)[[\epsilon]](u^{-1}), d).
\]
It is defined as 
\[
\Phi(a_0\otimes a_1\otimes \dots \otimes a_k) 
=
(e^{-u\iota_\omega}\tau_{2n})(Sh( \hat{a}_0\otimes \dots \otimes \hat{a}_k, 1\otimes e^{\wedge A}))
\]
where $a_0,..,a_k\in \WA_\epsilon$. The notation is as follows: The flat section $\hat{a}_j$ of $\sh{F}$ is the flat lift of $a_j$. $Sh$ is the shuffle product:
\begin{rem}
For an algebra $A$, the shuffle product $Sh:C_\bullet(A,A)\times C_\bullet(A,A) \rightarrow C_\bullet(A,A)$ is the bilinear map defined by
\[
Sh(a_0\otimes a_1\otimes\dots\otimes a_k, a_0'\otimes a_{k+1}\otimes\dots\otimes a_{k+l}) 
=
\sideset{}{'}\sum_\sigma sgn(\sigma) (a_0a_0')\otimes a_{\sigma(1)}\otimes\dots \otimes a_{\sigma(k+l)}
\]
where the sum is over all permutations $\sigma$ such that $\sigma^{-1}(1)<\sigma^{-1}(2)<\cdots <\sigma^{-1}(k)$ and $\sigma^{-1}(k+1)<\sigma^{-1}(k+2)<\cdots <\sigma^{-1}(k+l)$.
\end{rem}
Furthermore, $e^{\wedge A}$ is shorthand for the differential form with values in $T^\bullet (F)$ such that
\[
e^{\wedge A} (\xi_1,..,\xi_l) = A(\xi_1)\wedge \dots A(\xi_l)
\]
for $l=0,1,..$ and $\xi_1,..,\xi_l$ vector fields on $M$.
Finally the cocycle $(e^{-u\iota_\omega}\tau_{2n})$ is applied fiberwise.

\begin{cor}
The map $\Phi$ is a chain map.
\end{cor}
\begin{proof}
By Theorem \ref{thm:tauw} we already know that $(e^{-u\iota_\omega}\tau_{2n})\circ(b+u^{-1}B')=0$. Hence it is sufficient to show that 
\[
(b+u^{-1}B'-\nabla)Sh(\hat{a}_0\otimes \dots \otimes \hat{a}_k, 1\otimes e^{\wedge A}) =Sh((b+u^{-1}B')(\hat{a}_0\otimes\dots \otimes \hat{a}_k), 1\otimes e^{\wedge A}).
\]
The $(u^{-1})^0$-part of the equality has already been shown by FFS. For the $u^{-1}$-part, fix $l\in \{0,1,..\}$, let $\xi_1,..,\xi_l$ be vector fields and define $\hat{a}_{k+j}:=A(\xi_j)$. Note that 
\[
Sh(\hat{a}_0\otimes \dots \otimes \hat{a}_k, 1\otimes e^{\wedge A})(\xi_1,..,\xi_l)
= 
\sideset{}{''}\sum_\sigma sgn(\sigma) \hat{a}_0 \otimes \hat{a}_{\sigma(1)}\otimes\dots \otimes \hat{a}_{\sigma(k+l)}
\]
where the sum is over all permutations $\sigma$ such that $\sigma^{-1}(1)<\sigma^{-1}(2)<\cdots <\sigma^{-1}(k)$. Furthermore 
\begin{align*}
&B'( Sh(\hat{a}_0\otimes \dots \otimes \hat{a}_k, 1\otimes e^{\wedge A})(\xi_1,..,\xi_l) ) \\
&= 
\sideset{}{'''}\sum_\sigma sgn(\sigma) \hat{a}_0 \otimes \hat{a}_{\sigma(1)}\otimes\dots \otimes \hat{a}_{\sigma(k+l)} \\
&= Sh(B' (\hat{a}_0\otimes \dots \otimes \hat{a}_k), 1\otimes e^{\wedge A})(\xi_1,..,\xi_l)
\end{align*}
where the sum is over all permutations $\sigma$ such that for some $1\leq m \leq k$: $\sigma^{-1}(m)<\sigma^{-1}(m+1)<\cdots <\sigma^{-1}(k)<\sigma^{-1}(1)< \cdots <\sigma^{-1}(m-1)$.
\end{proof}

The map $\Phi$ induces a map on homology, which we call $[\Phi]: HC^{per}_\bullet(A_\epsilon)\rightarrow H^\bullet(M)[[\epsilon]]((u))$. 

\subsection{The $\hat{A}Ch$-class of a symplectic manifold.}
We reuse the setting of the previous subsection.
There is a ``canonical'' class in the periodic cyclic homology $HC^{per}_\bullet(A_\epsilon)$, namely that of the constant function $1$.

\begin{prop}
\label{prop:gtauisach}
Let $F$ be the curvature of the symplectic connection $\nabla$ and let $\Omega = F + \nabla A +\frac{\epsilon^{-1}}{2}\co{A}{A}$ be the characteristic class of the star product $\star$ (aka. the Fedosov curvature of $D$). Then
\[
[\Phi(1)] = u^n [\hat{A}(-\epsilon u^{-1} F) Ch(u^{-1}\epsilon \Omega) ]
\]
\end{prop}
\begin{proof}
Most of the proof is a copy of section 5.7 in \cite{felder}. First compute $C(A,A)$, where $C$ is as in section \ref{sec:lacohomcc}. We can assume that $A$ has no quadratic part as any quadratic part could be incorporated into the connection $\nabla$, and hence $pr(A)=0$. Then 
\[
\frac{1}{2}C(A,A) = -\frac{1}{2}pr(\co{A}{A}) = \epsilon pr( F- \Omega + \nabla A)=\epsilon ( F-\Omega).
\]
For the last equality, it was used that $\nabla A$ has no quadratic part, since $\nabla$ preserves degrees.

Now apply Theorem \ref{thm:tau2kisach} to obtain that 
\begin{align*}
[\Phi(1)]_{2k} &= [u^{n-k} \tau_{2k}^r(1\otimes A\wedge \dots\wedge A)] \\
&=
[(-1)^ku^{n-k}\chi((\hat{A}Ch)_{k})(A\wedge \dots\wedge A)] \\
&=
[(-1)^k u^{n-k}(\hat{A}Ch)_{k}(\epsilon ( F-\Omega),..,\epsilon ( F-\Omega))] \\
&=
[(-1)^k u^{n-k}\hat{A}(\epsilon F) Ch(-\epsilon \Omega)]_{2k} \\
&=
[u^n\hat{A}(-\epsilon u^{-1} F) Ch(\epsilon u^{-1} \Omega)]_{2k}
\end{align*}

\end{proof}

\subsection{The $TdCh$-class of a holomorphic bundle}
\label{sec:tdch}
There is an application of the above results to complex manifolds and holomorphic bundles. To present it in the right context, we first need to bring some recent results of Engeli and Felder to the attention of the reader.
\subsubsection{The generalized Riemann-Roch-Hirzebruch Theorem}
Let $E\rightarrow M$ be a holomorphic vector bundle over the compact complex $n$-dimensional manifold $M$. Let $\sh{E}$ be the sheaf of holomorphic sections of $E$. Let $\sh{D}$ be the sheaf of holomorphic differential operators on $E$. Any global differential operator $D\in \Gamma(M,\sh{D})=H^0(\sh{D})$ acts on the sheaf cohomology $H^\bullet(\sh{E})$ of $\sh{E}$. It turns out that there is an explicit integral formula computing the (super-)trace of this action:
\begin{equation}
\label{equ:felderformula}
\str_{H^\bullet(\sh{E})}(D) = \int_M \mu(D).
\end{equation}

Here $\mu(D)$ is some differential $2n$-form, for which an explicit formula can be written down using the cocycle $\tau_{2n}^r$. The above formula has essentially been conjectured by Feigin, Losev and Shoikhet \cite{FLS}, and has been proved by Engeli and Felder \cite{engeli}. See also the work of Ramadoss \cite{ramadoss}.

The details are as follows. Consider the ``sheaf of cyclic chains'' $\sh{CC}_{\bullet}$ of the sheaf of algebras $\sh{D}$. It is defined as the sheaf associated to the presheaf 
\[
\sh{CC}_{\bullet}(U) = CC_\bullet(\sh{D}(U))
\]
for open $U\subset M$. Here the cyclic complex $CC_\bullet(\cdot)$ is defined using projectively completed tensor products instead of the usual algebraic ones (see \cite{connes}, section 5). The sheaf $\sh{CC}_{\bullet}$ is equipped with a grading and a differential $(b+u^{-1}B')$ coming from the grading and differential on the cyclic chain complex. Denote by $\mathbb{H}^\bullet(\sh{CC})$ the hypercohomology of the differential graded sheaf $\sh{CC}_{\bullet}$. There is a natural map from the cyclic homology of the algebra $H^0(\sh{D})$ of global holomorphic differential operators on $E$ to $\mathbb{H}^\bullet(\sh{CC})$. It is given by the following composition:
\[
CC_\bullet(H^0(\sh{D})) \rightarrow H^0(\sh{CC}_{\bullet}) \rightarrow \mathbb{H}^\bullet(\sh{CC}).
\]

In the Hochschild case ($W=\gf$) the hypercohomology $\mathbb{H}^\bullet(\sh{CC})$ can be computed by calculations in \cite{ramadoss} and \cite{engeli} (section 3) and shown to agree with the de Rham cohomology of $M$:
\begin{equation}
\label{equ:hypcisdrh}
\mathbb{H}^\bullet(\sh{CC}) \cong H^{2n-\bullet}(M).
\end{equation}

Similarly, for periodic cyclic homology ($W=\C(u^{-1})$)
\begin{equation}
\label{equ:hypcisdr}
\mathbb{H}^\bullet(\sh{CC}) \cong H^{2n-\bullet}(M)(u^{-1})
\end{equation}

Hence, by composition we obtain maps
\begin{align*}
\Psi:CC_\bullet(H^0(\sh{D})) &\rightarrow H^{2n-\bullet}(M)  &\quad \text{Hochschild case} \\
\Psi:CC_\bullet(H^0(\sh{D})) &\rightarrow H^{2n-\bullet}(M)(u^{-1}) &\quad \text{Periodic cyclic case}.
\end{align*}

These maps are interesting for the following reason:
\begin{thm}[Engeli, Felder \cite{engeli}]
\label{thm:ef}
In the Hochschild case ($W=\C$) the degree zero component of $ev_1$
\[
\Psi_0 : C_0(H^0(\sh{D}))\cong H^0(\sh{D}) \rightarrow H^{2n}(M) \cong \C
\]
is exactly the supertrace on the sheaf cohomology of $\sh{E}$ that appears on the left of \eref{equ:felderformula}. 
\end{thm}

The cocycle $\tau_{2n}^r$ enters into this story in so far, that it allows for writing down an explicit chain map inducing the isomorphism \eref{equ:hypcisdr} on cohomology. Composing with the map $\int_M: H^{2n}(M) \stackrel{~}{\to} \C$ yields the r.h.s. of \eref{equ:felderformula}. 

Let now $1\in H^0(\sh{D})$ be the identity operator on $E$. The Riemann-Roch-Hirzebruch Theorem tells us that $\str_{H^\bullet(\sh{E})}(1)=\int_M Td(M)Ch(E)$ is the integral over the Todd-class of $M$ times the Chern-character of $E$. By looking at Theorem \ref{thm:ef} and the fact that $H^{2n}(M) \cong \C$ we hence know that, in the Hochschild case
\[
\Psi_0(1) = \left[Td(M)Ch(E)\right]_{2n} \in H^{2n}(M).
\]

The contribution of this paper is to extend this result to the cyclic case.
\begin{prop}
\label{prop:tauistdch}
For periodic cyclic homology ($W=\C(u^{-1})$) we obtain that 
\[
\Psi(1) = [Td_u(M)Ch_u(E)] \in H^\bullet(M)(u^{-1}).
\]
Here, the subscript $u$ means that the classes can be represented by the forms $Ch_u(E)=\tr(\exp(-u^{-1}F))$ and
\[
Td_u(M) = \det \left(\frac{u^{-1}R}{e^{u^{-1}R}-1} \right)
\]
where $R$ and $F$ are the curvatures of connections on $T^{1,0}M$ and $E$ respectively. 
\end{prop}
\begin{rem}
The proof will show that the above equality holds already on chains, not just in cohomology, if one defines $\Psi$ appropriately on chains. This is in contrast to the symplectic case treated in the previous section, where the analogous statement is true only in cohomology.
\end{rem}

\subsubsection{Complex formal geometry}
\label{sec:cmplxformal}
In order to prove Proposition \ref{prop:tauistdch}, we need to recall some basic notions of formal geometry on complex manifolds. The results will be analogous to Fedosov's theory on real symplectic manifolds. On the general theory, we will be rather brief and refer the interested reader to \cite{calaque} for details.

Let $T'M=T^{1,0}M$ and $T^{*'}M=T^{*(1,0)}$ be the holomorphic tangent and cotangent bundles of $M$.
Let $\sh{F}=\prod_i S^i T^{*'}M$ be the bundle of formal fiberwise power series. In local coordinates $x^i$ its sections are formal power series $\sum_{K} c_K(x) y^K$, where the sum runs over all multiindices $K$, the $c_K(x)$ are coefficient functions and $y^j=dx^j$ form a basis of $T_x^{*'}M$. Similarly, one can define the bundle $V=Der(\sh{F})$ of formal fiberwise vector fields and the bundle $D$ of fibrewise differential operators (on $M$). The bundle $D_E=D\otimes \End(E)$ is the bundle of formal fiberwise differential operators on $E$.

Pick connections $\nabla_0$ on $T'M$ and $\nabla_{E,0}$ on $E$ respectively. We require that $\nabla_0$ is torsion-free. These connections can naturally be extended to the bundles $\sh{F}$, $V$, $D$ and $D_E$. We denote the extension to $D_E$ by $\nabla_{D_E,0}$ etc. It turns out that one can find a form $A \in \Gamma(M, (V\oplus \sh{F}\otimes End(E))\otimes T^*M)$ with values in the first order formal fiberwise differential operators such that  the connection 
\[
\nabla_{D_E} = \nabla_{D_E,0} + \co{A}{\cdot}
\]
is flat. Concretely, this means that the Maurer-Cartan equation
\begin{equation}
\label{equ:cmplxMC}
\nabla_{D_E,0}A + \frac{1}{2}\co{A}{A} - \mathbb{1}_E\otimes R^j_k y^k \pderi{}{y^j} + F = 0
\end{equation}
is satisfied. Here the two-forms $R^j_k$ are the components of the curvature form of $\nabla_0$ and $F$ is the curvature form of $\nabla_{E,0}$, with values in $\End(E)$.

The fact of central importance is now that the sheaf of algebras given by the flat sections of $D_E$ (w.r.t. the flat connection $\nabla_{D_E}$) is isomorphic to the sheaf of algebras of holomorphic differential operators on $E$. In particular, given a global holomorphic differential operator $a\in \Gamma(M,\sh{D}_E)$, one can define the unique flat lift $\hat{a}\in \Gamma_{flat}(M,D_E)$. 

\subsubsection{An explicit formula for $\Psi_0$}
To define $\Psi_0$, we need one last ingredient: The polynomial Weyl algebra has two alternative definitions, namely as differential operators on $\R^n$ with the usual product, and as polynomials on $\R^{2n}$ with product the Moyal product w.r.t. the standard symplectic structure $\omega$. Furthermore, the complexified algebra of polynomial differential operators on $\R^n$ is isomorphic to the algebra of (holomorphic) polynomial differential operators on $\C^n$. 

Denote by $Q$ the map that associates to any holomorphic polynomial differential operator the corresponding polynomial.
Then we can define
\[
\Psi_0(a) = (e^{-u \iota_\omega} \tau_{2n}^r) (Q(\hat{a})\otimes Q(A)\wedge..\wedge Q(A)).
\]
Here we silently ``cheated'' twice: In general $\hat{a}$ is not polynomial, but a (fiberwise) formal power series. Hence we have to extended $Q$ to non-polynomial differential operators and $\tau_{2n}^r$ to power series. This is possible, since the total number of derivatives $\pderi{}{z}$ occuring in the $\hat{a}$ and the $A$'s is finite. Hence $Q(\hat{a})$ $Q(A)$ will be polynomials in the $p$'s, with coefficients in the power series in $q$'s. But terms with more $q$'s than $p$'s never contribute by $\alg{gl_n}$-invariance. Hence there are only finitely many nonvanishing terms and convergence problems do not arise.

\subsubsection{The Proof of Proposition \ref{prop:tauistdch}}
The proof is very similar to that of Proposition \ref{prop:gtauisach}. We note that
\begin{align*}
\frac{1}{2} C(Q(A),Q(A)) &= -\frac{1}{2} pr(Q(\co{A}{A})) =  pr(Q(\nabla_{D_E,0}A-R^j_iy^i \pderi{}{y^j}+F)) \\
&=
-R^j_i p_jq^i + \frac{1}{2}\tr(R) +F.
\end{align*}
For the second equality we used the Maurer-Cartan equation \eref{equ:cmplxMC}.

Before continuing, a remark about traces is in order. We have the standard inclusion $\alg{gl}_n\subset \alg{sp}_{2n}$. Define the traces $\tr_{\alg{gl}_n}$ and $\tr_{\alg{sp}_{2n}}$ on their universal enveloping algebras by the trace in the defining representations. Then, for $x\in \alg{gl}_n$ and $j=0,1,..$ we have
\[
\tr_{\alg{sp}_{2n}}(x^j) = \left( 1+(-1)^j\right)\tr_{\alg{gl}_n}(x^j).
\]
Hence the following identity holds:
\[
\hat{A}_{\alg{sp}_{2n}}(x) := e^{-\sum_{j\geq 1} \frac{B_j}{4j (2j)!}\tr_{\alg{sp}_{2n}}(x^{2j}) }
=
e^{-\sum_{j\geq 1} \frac{B_j}{2j (2j)!}\tr_{\alg{gl}_{n}}(x^{2j}) }
=: \hat{A}_{\alg{gl}_{n}}^2(x).
\]

Taking this into account and using Theorem \ref{thm:tau2kisach}, we obtain that
\begin{align*}
(\Psi(1))_{2k}
&=
(-1)^k\left( \hat{A}_{\alg{sp}_{2n}}(-R) Ch(\frac{1}{2}\tr(R)+F) \right)_{2k} \\
&=
(-1)^k\left( \hat{A}_{\alg{gl}_{n}}^2(-R) Ch(\frac{1}{2}\tr(R)+F) \right)_{2k} \\
&=
(-1)^k\left( Td(-R) Ch(F) \right)_{2k} 
= \left( Td(R) Ch(-F) \right)_{2k}.
\end{align*}

\appendix

\section{Standard definitions}
\label{app:stddef}
\subsection{Hochschild and cyclic (co)homology}
Let $A$ be any algebra with unit $\mathbb{1}$ over $\C$. Let $\bar{A}=A/ \mathbb{1}\cdot \C$. The normalized Hochschild chain complex with coefficients in the $A$-bimodule $M$ is defined as 
\[
C_\bullet(A,M) = M\otimes \bar{A}^{\otimes \bullet}
\]
with differential $b$ defined such that
\begin{multline*}
b(m\otimes a_1\otimes \dots \otimes a_n) 
= m\cdot a_1\otimes a_2\otimes \dots \otimes a_n + \\
+\sum_{j=1}^{n-1} (-1)^j m\otimes a_1\otimes \dots \otimes a_j a_{j+1} \otimes a_n
+(-1)^n a_n\cdot m\otimes a_1\otimes \dots \otimes a_{n-1}.
\end{multline*}

Similarly, the Hochschild cochain complex with values in $M$ is 
\[
C^\bullet(A,M) = \hom(\bar{A}^{\otimes \bullet},M)
\]
with differential $d$ defined as 
\begin{multline*}
(d\phi)(a_1\otimes \dots \otimes a_n) 
= a_1\cdot \phi(a_2\otimes \dots \otimes a_n) + \\
+\sum_{j=1}^{n-1} (-1)^j \phi( a_1\otimes \dots \otimes a_j a_{j+1} \otimes a_n )
+(-1)^n \phi( a_1\otimes \dots \otimes a_{n-1})\cdot a_n.
\end{multline*}

Now let $M=A$. Then there is another differential $B'$ of degree $+1$ on $C_\bullet(A,M)$. It is defined by
\[
B'(a_0\otimes a_1\otimes \dots \otimes a_n)
=
\sum_{j=0}^n (-1)^{nj} 1\otimes a_j\otimes a_{j+1}\otimes \dots \otimes a_{j-1}.
\]
Similarly, the adjoint operator to $B'$ defines a differential $B$ on the Hochschild cochain complex $C^\bullet(A,A^*)$.

One can check that $b$ and $B'$ and $d$ and $B$ anticommute. Let $W$ be some $\C[[u]]$-module for $u$ a formal variable of degree $+2$. The cases of interest in this paper are $W=\C$ the trivial module and $W=\C(u)$ the formal Laurent series. Consider the complex 
\[
CC^\bullet(A) = C^\bullet(A,M)\tW
\]
and equip it with the differential $d+uB$. This complex is called the cyclic complex and its cohomology the cyclic cohomology of $A$. Similarly, for a formal variable $u$ of degree $-2$ one can define
\[
CC_\bullet(A) = C_\bullet(A,M)\tW
\]
with differential $b+uB'$, the cyclic cochain complex. In this paper, in order to not confuse the $u$'s of the chain and cochain complexes, we use $u$ of degree $+2$ as the formal variable for the cochain complex and $u^{-1}$ (degree $-2$) for the chain complex.

\subsection{More structures}
\label{app:morestruct}
Define the following operators on the Hochschild cochain complex, for $a\in A$ :

\begin{align*}
\iota_a : C^\bullet(A,M) \to C^{\bullet-1}(A,M) \\
(\iota_a \phi)(a_0,..,a_n) = \sum_{j=0}^n (-1)^j \phi(a_0,..,a_j,a,a_{j+1},..,a_n).
\end{align*}
where $\phi\in C^{n+1}(A,M)$ and $a_0,..,a_n\in A$.

Define further $L_a:= \aco{d}{\iota_a}$. One can check the following explicit formula for $L_a$:
\begin{multline*}
(L_a\phi)(a_1,..,a_n) = a\cdot \phi(a_1,..,a_n) - \phi(a_1,..,a_n) \cdot a + \sum_{j=1}^n \phi(a_1,..,\co{a_j}{a},..,a_n).
\end{multline*}

\begin{lemma}
\label{lem:morestruct}
The following list of identities holds on $C^\bullet(A,M)$ for any $a, b\in A$ and any $A$-bimodule $M$:
\begin{align*}
\co{d}{L_a} &= 0 
& \co{L_a}{\iota_b} &= \iota_{\co{b}{a}} \\
\co{L_a}{L_b} &= L_{\co{a}{b}}.
\end{align*}
If $M=A^*$, then furthermore
\begin{align*}
\aco{\iota_a}{B} &= 0 
&\aco{L_a}{B} &= 0.
\end{align*}
\end{lemma}
\begin{proof}
The first equation is an immediate consequence of the definition.

The proof of the second equality is a direct calculation
\begin{align*}
(L_a \iota_b \phi)(a_1,..,a_n) 
&= \sum_{ j=0}^n (-1)^j \sum_{k=1}^n\phi(a_1,..,a_j,b,a_{j+1},..,\co{a_k}{a},..,a_n) +\\
   &\quad + \co{a}{ (\iota_b\phi(a_1,..,a_n))} \\
&=(\iota_b L_a)\phi(a_1,..,a_n) - \sum_{ j=0}^n (-1)^j\phi(a_1,..,a_j,\co{b}{a},a_{j+1},..,a_n) \\
&= (\iota_b L_a)\phi(a_1,..,a_n) + (\iota_{\co{a}{b}}\phi)(a_1,..,a_n)
\end{align*}
where $\phi \in C^{n+1}(A,M)$ and $a_1,..,a_n\in A$.

The third equality follows from the second
\begin{align*}
\co{L_a}{L_b} &= \aco{\co{L_a}{d}}{\iota_b} + \aco{d}{\co{L_a}{\iota_b}} \\
&= 0 + \aco{d}{\iota_{\co{a}{b}}} = L_{\co{a}{b}}
\end{align*}

For the fourth equality, we calculate
\begin{align*}
(\iota_a B \phi)(a_0,..,a_n) 
&=
\sum_{ j=0}^{n}(-1)^j \left( 
\sum_{k=0}^{j} (-1)^{(n+1)k} \phi(1,a_k,..,a_{j},a,..,a_{k-1}) \right.\\
  &\quad - \sum_{ k=j+1}^{n}(-1)^{(n+1)k} \phi(1,a_k,..,a_{j},a,..,a_{k-1}) + \\
  &\quad + \left.  (-1)^{(j+1)(n+1)} \phi(1,a, a_{j+1},..,a_{j}) \right) \\
&=
\sum_{ j=1}^{n+1}
\left( \sum_{k=0}^{n+1-j} (-1)^{k+j-1+k(n+1)} - \sum_{ k=n+2-j}^{n}(-1)^{k+j-1-n-1+(k+1)(n+1)} \right) \\
  &\quad \phi(1,a_k,..,a_{k+j-1},a,a_{k+j},..,a_{k-1}) +\\
  &\quad +
\sum_{k=0}^n (-1)^{k(n+1)+k-1}
\phi(1,a, a_{k},..,a_{k-1})
\\
&= 
\sum_{ j=1}^{n+1}
\sum_{k=0}^{n} (-1)^{j-1+kn}
\phi(1,a_k,..,a_{k+j-1},a,a_{k+j},..,a_{k-1})
+ \\ 
&\quad +
\sum_{k=0}^n (-1)^{kn+1}
\phi(1,a, a_{k},..,a_{k-1})
\\
&= -\sum_{ j=0}^{n+1} (-1)^j \sum_{k=0}^n (-1)^{nk} \phi(1,a_k,..,a_{k+j-1},a,a_{k+j},..,a_{k-1}) \\
&= -(B \iota_a \phi)(a_0,..,a_n).
\end{align*}
In the second line we reordered the sum such that terms with $a$ inserted into the same ``slot'' of $\phi$ stand together.

The last equality in the proposition is an easy consequence of the fourth and the fact that $\aco{d}{B}=0$.
\end{proof}

Obviously, the operator $\iota_a$ can be $u$-linearly extended to the cyclic cochain complex $CC^\bullet_W(A)$.
The above discussion then shows that the complex $CC_W^\bullet(A)$ carries the structure of a differential graded $A$-space, where $A$ is seen as a Lie algebra. Let us recall the definition:
\begin{defi}
Let $\alg{g}$ be a Lie algebra. A \emph{differential graded $\alg{g}$-space} is a graded vector space $V^\bullet$ with differential $d$, a left $\alg{g}$-action denoted $L$ on $V^\bullet$, and an operation 
\begin{align*}
\iota : \alg{g} \otimes V^\bullet &\rightarrow V^{\bullet-1} \\
x \otimes v &\mapsto \iota_x v
\end{align*}
subject to the following conditions
\begin{itemize}
\item $\co{\iota_x}{\iota_y}=0$ for all $x,y\in \alg{g}$.
\item $L_x = \co{d}{\iota_x}$.
\item $\co{L_x}{\iota_y} = \iota_{\co{x}{y}}$.
\end{itemize}
\end{defi}

\subsection{Lie algebra cohomology}
Let $\alg{g}$ be a Lie algebra and $\alg{h}\subset \alg{g}$ a subalgebra. Let $M$ be a $\alg{g}$-module. The relative Chevalley-Eilenberg (cochain-) complex with values in $M$ is defined as
\[
C^\bullet(\alg{g},\alg{h};M) = \Hom_{\alg{h}}(\Lambda^\bullet(\alg{g}/\alg{h}), M).
\]
We will think of chains as elements of $\Hom_{\alg{h}}(\wedge^\bullet\alg{g}, M)$ that vanish on the ideal $I$ generated by $\alg{h}$ in the algebra $\Lambda^\bullet \alg{g}$.

It carries a differential $d$ such that
\begin{multline*}
(d\phi)(x_1\wedge \dots \wedge x_n) 
= \sum_{j=1}^n (-1)^{j+1} x_j \cdot \phi(x_1\wedge \dots \wedge \hat{x}_j \wedge\dots \wedge x_n) + \\
\sum_{1\leq i<j\leq n} (-1)^{i+j} \phi(\co{x_i}{x_j}\wedge \dots \wedge \hat{x}_i \wedge\dots \wedge \hat{x}_j \wedge\dots \wedge x_n).
\end{multline*}

One needs to show that this map is well defined, i.e., that $d\phi$ vanishes on the ideal $I$ if $\phi$ does. For this, assume that $x_1=h\in \alg{h}$. Then 
\begin{multline*}
(d\phi)(h\wedge x_2\wedge \dots \wedge x_n) 
= h \cdot \phi(x_1\wedge \dots \wedge \hat{x}_j \wedge\dots \wedge x_n) + \\
\sum_{2\leq j\leq n} (-1)^{j+1} \phi(\co{h}{x_j}\wedge \dots \wedge \hat{x}_j \wedge\dots \wedge x_n) \\
= (h\cdot \phi)(x_2\wedge \dots \wedge x_n) = 0.
\end{multline*}

The cohomology of the above complex is called the relative Lie algebra cohomology $H^\bullet(\alg{g},\alg{h};M)$, and in the special case $\alg{h}=\{0\}$ simply the Lie algebra cohomology $H^\bullet(\alg{g};M)$ of $\alg{g}$.

\section{Relation to Tsygan Formality}
\begin{figure}
\includegraphics[width=.45\textwidth]{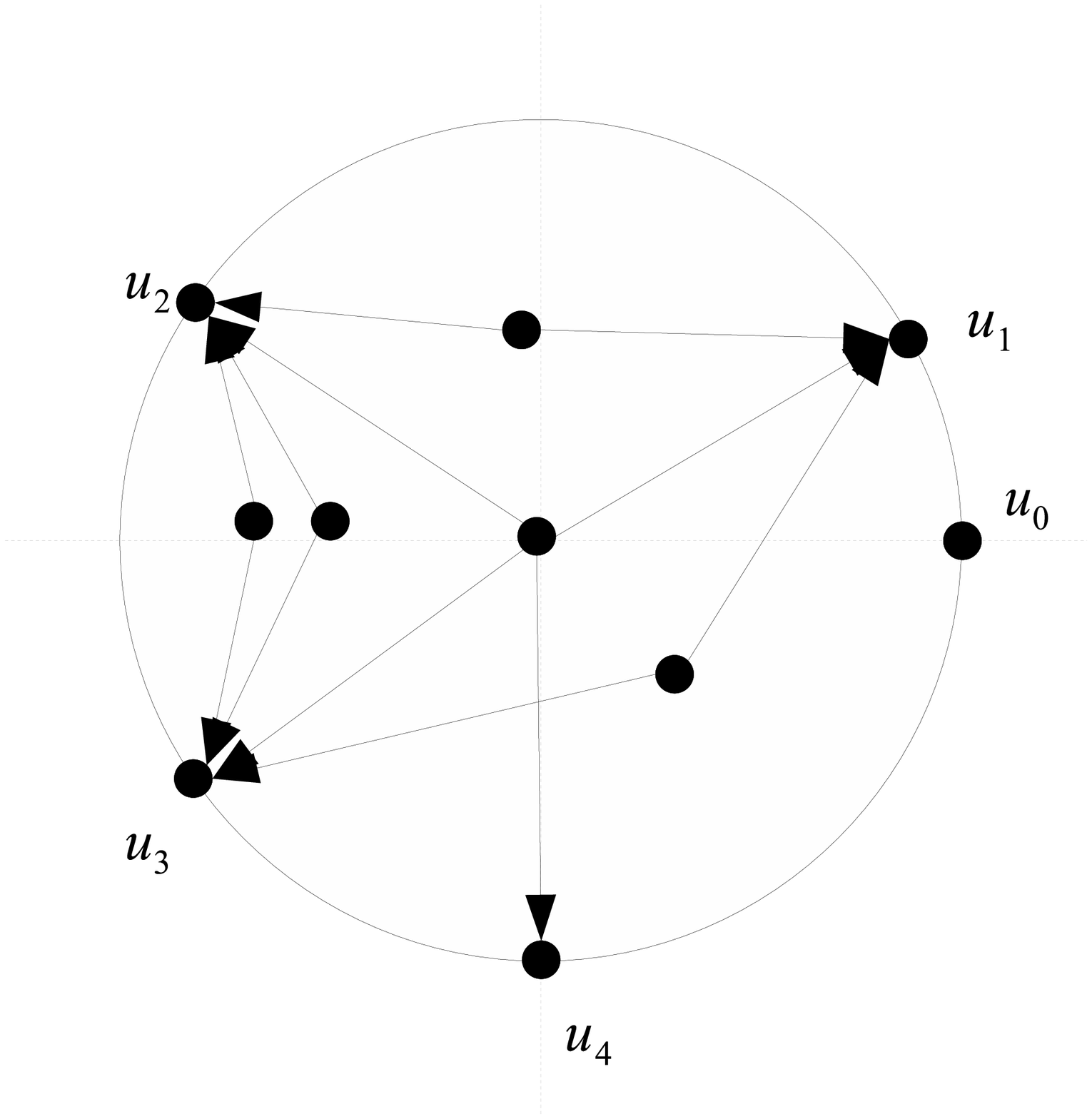}
\includegraphics[width=.45\textwidth]{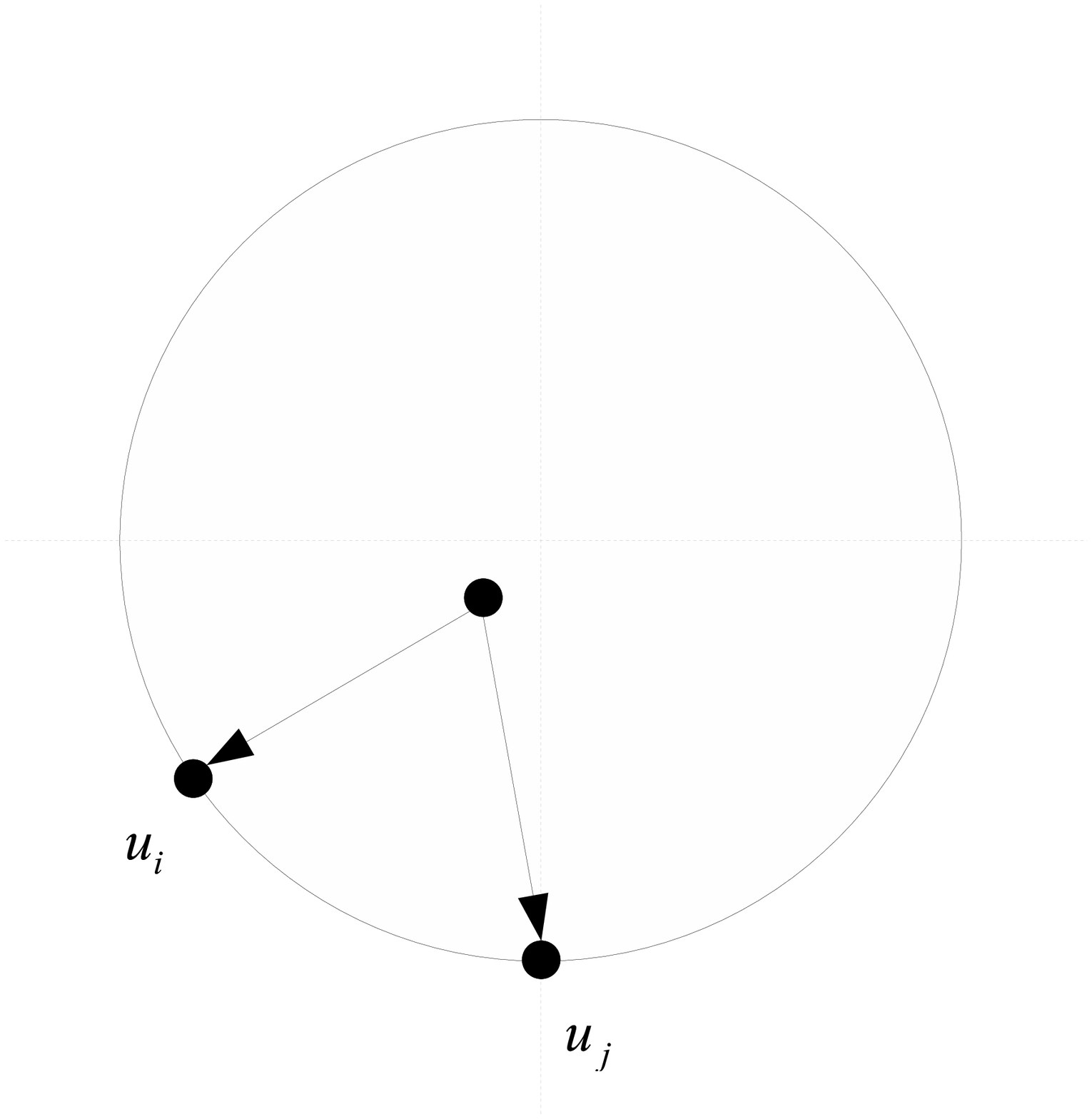}
\caption{\label{fig:shgraphs} On the left, a Shoikhet graph contributing to $\tau_4$ is shown. It is composed out of primitive graphs like the one on the right.}
\end{figure} 

The cocycles $\tau_{2n}$ and $\tau_w$ both can be derived from Tsygan's formality conjectures, which are by now theorems. For the Hochschild case, the relation has been explained briefly by FFS, without mentioning details. 
We explain the relation briefly in this section, but assume that the reader is already familiar with the Tsygan conjectures \cite{tsygan} and their proofs \cite{shoikhet}, \cite{me}. We apologize for being sketchy.

Basically, the cocycle $\tau_w$ can be obtained by localizing Tsygan's $L_\infty$-morphism of modules at the constant Poisson bivector field $\omega^{-1}$ on $\R^{2n}$. The resulting differential on the space of forms is then $u^{-1}d+L_{\omega^{-1}}$, where $L_{\omega^{-1}}$ is the Lie derivate wrt. the Poisson bivector field. As noted by Tsygan, the operator $e^{-u\iota_\omega}$ intertwines this differential with the differential $u^{-1}d$, with respect to which the cohomology is trivially computed, and is concentrated in form degree $0$.
 
Concretely, the components $\tau_{2k}$ can be written as a sum of Shoikhet graphs with $2k+1$ external vertices. We put these vertices at positions $0=u_0<u_1<\cdots u_{2k}\leq 1$ on the circle of unit circumference. An example graph is shown in Figure \ref{fig:shgraphs} (left), from which the general case should be clear. Let us compute the weights of the graphs, using the following Lemma.
\begin{lemma}
The right graph in Figure \ref{fig:shgraphs} has weight $-B_1(u_j-u_i)$.
\end{lemma}
\begin{proof}
Denote by $C$ the configuration space of one point ($z$) in the interior, and one point ($u$) on the boundary, that is allowed to move between $u_i$ and $u_j$. Then by Stokes
\begin{align*}
0 &= \int_{C}d(d\phi(z,u_i)d\phi(z,u)) = \int_{\partial C}d\phi(z,u_i)d\phi(z,u) \\
&= -\frac{1}{2} + \int_{C\cap \{u=u_j\} } d\phi(z,u_i)d\phi(z,u_j) + \int_{u=u_i}^{u_j}du 
\end{align*}
Here the first term is the contribution for the stratum on which $z$ and $u$ approach $u_i$. The second is the stratum where $u$ approaches $u_j$. It conincides with the desired weight. The third is the stratum where $z$ approaches the center of the disk. Hence one obtains the desired result.
\end{proof}

All graphs connecting $u_i$ with $u_j$ thus contribute the operator $\exp(-B_1(u_j-u_i)\alpha_{ij})$. The edges connecting the central vertex to the external vertices yield a contribution (note that $u^k(\omega^{-1})^{\wedge k}$ is inserted there) $u^k\pi_{2k}du_1\wedge \dots du_{2k}$. Putting everything together yields the desired formula:
\[
\int_{0=u_0<u_1<\cdots u_{2k}\leq 1}du_1\wedge \dots du_{2k} \prod_{i<j} \exp(B_1(u_j-u_i)\alpha_{ji}) \pi_{2k}.
\]

\section{Direct proof of Theorem \ref{thm:tau2kisach}}
\label{app:directproof}
We give here an alternative proof of Theorem \ref{thm:tau2kisach}, without using that the special case $k=n$ was already proved by FFS. In fact, we will prove Proposition \ref{prop:tauisachonwn} by showing equation \eref{equ:tauonwn} directly, i.e., by computing the integrals involved.

As in section \ref{sec:cctheproof} we can assume that
\[
v_j = 
\begin{cases}
p_{\alpha_j} q_{\beta_j} q_{\gamma_j} \otimes \mathbb{1}_{r\times r} & \quad \text{for $j=1,..,m$ }\\
q_{\alpha_j} \otimes M_j & \quad \text{for $j=m+1,..,k$ } \\
p_{\alpha_j} \otimes \mathbb{1}_{r\times r} & \quad \text{for $j=k+1,..,2k$ }
\end{cases}
\]
where the $v_j$ are those in \eref{equ:tauonwn}. 

The left hand side of \eref{equ:tauonwn} can be simplified by using (the straightforward analog of) Lemma \ref{lem:22}:
\begin{multline*}
\tau_{2k}(1\otimes (v_1 \wedge \dots \wedge v_{2k}) \\
=
\mu_{2k} \int_{[0,1]^{2k}} du_1\cdots du_{2k}
\prod_{0\leq i< j \leq 2k} e^{b_1(u_j-u_i)\alpha_{ji}} 
\pi_{2k}(1\otimes v_1 \otimes \dots \otimes v_{2k}) 
\end{multline*}

Next, note that by the assumptions on the $v_j$ above:
\begin{multline*}
\pi_{2k}(1\otimes v_1 \otimes \dots \otimes v_{2k}) = \\
=(-1)^{k(k+1)/2}\sum_{\sigma}sgn(\sigma) 1\otimes C(v_1,v_{k+\sigma(1)}) \otimes \dots \otimes C(v_k,v_{k+\sigma(k)})\otimes 1\otimes\cdots \otimes 1
\end{multline*}
where $C$ is the ``curvature'' defined in \eref{equ:curvdef} of section \ref{sec:lacohomcc}. Hence we can integrate out variables $u_{k+1},..,u_{2k}$ and arrive at a sum of expressions of the form
\[
H_k(a_1,..,a_k)=
\mu_{k} \int_{[0,1]^{k}} du_1\cdots du_k
\prod_{1\leq i< j \leq k} e^{b_1(u_j-u_i)\alpha_{ji}} 
\tr(1\otimes a_1 \otimes \dots \otimes a_k)
\]
for $a_j\in \alg{h}_1= \alg{gl}_n\oplus \alg{gl}_r$, $j=1,..,k$. In our case actually $a_1,..,a_m\in \alg{gl}_n$ and $a_{m+1},..,a_k\in \alg{gl}_r$. Then $H_k(a_1,..,a_k)= h_m(a_1,..,a_m) \tr(a_{m+1}\cdots a_k)$ where 
\[
h_m(a_1,..,a_m) 
=
\mu_{m} \int_{[0,1]^{m}} du_1\cdots du_m
\prod_{1\leq i< j \leq m} e^{b_1(u_j-u_i)\alpha_{ji}} 
1\otimes a_1 \otimes \dots \otimes a_m
\] 
 
Note that $h_m(a_1,..,a_m)$ is symmetric in the $a_1,..,a_m$. By polarization, it is hence sufficient to compute $h_m(x,..,x)$ for all $x\in \alg{gl}_n$.

\begin{prop}
For $x\in \alg{gl}_n$:
\[
h_m(x,..,x) = m! \hat{A}_m(x) 
\]
where $\hat{A}_m$ is the $m$-homogeneous component of the A-roof genus $\hat{A}\in (S^\bullet\alg{gl}_n)^{*\alg{gl}_n}$.
\end{prop}
\begin{proof}
Since $x$ is homogeneous quadratic, exactly two derivatives have to act on each $x$ in 
\begin{equation}
\label{equ:hmx}
h_m(x,..,x) 
=
\mu_{m}
\int_{[0,1]^{m}} du_1\cdots du_m
\prod_{1\leq i< j \leq m} e^{b_1(u_j-u_i)\alpha_{ji}} 
(1\otimes x \otimes \dots \otimes x).
\end{equation}
It follows that the exponentials have to be expanded only up to second order $e^{b_1(u_j-u_i)\alpha_{ji}}=1+b_1(u_j-u_i)\alpha_{ji}+(b_1(u_j-u_i)\alpha_{ji})^2+...$.
To each summand in the expansion of the resulting product one can associate a graph with vertex set $\{1,..,m\}$. The graph contains an edge from vertex $i$ to vertex $j$for each term $b_1(u_j-u_i)\alpha_{ji}$ occuring in the corresponding summand. Since exactly two derivatives have to act on each $x$, one easily sees that only those terms contribute, whose associated graphs are unions of cycles. To a cycle of length $j$ associate the term
\[
W_j = -\tr(x^j) \int_{[0,1]^{j}}du_1\cdots du_j b_1(u_1-u_2)b_1(u_2-u_3)\cdots b_1(u_j-u_1).
\]
Here $\tr(\cdot)= \tr_{\alg{sp}_{2n}}(\cdot)$.
Then the summand we started with is a product of $W_j$'s, one for each cycle of length $j$ in its associated graph (up to factors of $\frac{1}{2}$ coming from the expansion of the exponentials).

The integral on the right (``weight of the cycle graph'') can be evaluated to yield
\[
\int_{[0,1]^{j}}du_1\cdots du_j b_1(u_1-u_2)b_1(u_2-u_3)\cdots b_1(u_j-u_1)
 = (-1)^l \frac{B_l}{l!}
\]
where we used the notation and statement of Lemma \ref{lem:cycleweight} of Appendix \ref{app:bernoulli}.

Including the appropriate combinatorial factors we hence get
\begin{align*}
h_m(x,..,x) 
&=
m!
\sum_{\substack{j_2,j_3,.. \\ 2j_2+3j_3+..=m}}
\prod_l \frac{1}{j_l!} \left( \frac{W_{l}}{2 l} \right)^{j_l} \\
&=
m!
\sum_{\substack{j_2,.. \\ 2j_2+3j_3+..=m}}
\prod_l \frac{1}{j_l!} \left( \frac{-B_l\tr(x^l) (-1)^l}{(2l) l!} \right)^{j_l} 
\end{align*}
A term in this sum for fixed $j_2,j_3,..$ corresponds to the sum of all terms in \eref{equ:hmx} composed of $j_2$ cycles of length $2$, $j_3$ cycles of length $3$ etc. The combinatoric factors arise as follows:
\begin{itemize}
\item In the sum of graphs, graphs differing only by permuting whole cycles of the same length are identical, hence the factors $j_l!$ arise.
\item Graphs differing by cyclic permutations of vertices within each cycle are identical, hence the factors $l$ arise.
\item Also, reversing the order of vertices within each cycle does not change the graph, hence the factors of 2. For cycles of length 2 this order reversion coincides with the cyclic permutation above, however, there arises an additional factor of $\frac{1}{2}$ coming from the expansion of the exponential, so that we do not need to treat this case separately.
\end{itemize}

We can simplify this expression by summing over all $m$:
\begin{align*}
\sum_{m\geq 0} \frac{1}{m!} h_m(x,..,x)
= 
\exp\left( - \sum_{l\geq 2} \frac{(-1)^l B_l}{2 l l!} \tr(x^l) \right)
=
{\det}^{\frac{1}{2}}  \left( \frac{x/2}{\sinh(x/2)} \right) 
= \hat{A}(x)
\end{align*}

\end{proof}

Putting everything together, we can conclude the proof of Theorem \ref{thm:tau2kisach}:
\begin{align*}
&\tau_{2k}(1\otimes (v_1 \wedge \dots \wedge v_{2k}) \\
&=
(-1)^{k(k+1)/2}\sum_{\sigma\in S_k}sgn(\sigma) 
\hat{A}_m( C(v_1,v_{k+\sigma(1)}),\dots , C(v_m,v_{k+\sigma(m)}) \\
&\quad \quad \quad
Ch_{k-m}( C(v_{m+1},v_{k+\sigma(m+1)}),\dots , C(v_k,v_{k+\sigma(k)}) \\
&=
(-1)^k
\frac{1}{k!2^k}\sum_{\sigma\in S_{2k}}sgn(\sigma)
(\hat{A}Ch)( C(v_{\sigma(1)},v_{\sigma(2)}),\dots , C(v_{\sigma(2k-1)},v_{\sigma(2k)}) \\
&=
(-1)^k\chi(\hat{A}Ch)(v_1\wedge \dots \wedge v_{2k}).
\end{align*}

\section{Bernoulli Polynomials}
\label{app:bernoulli}
We here recall elementary facts about the Bernoulli polynomials.
The Bernoulli polynomials $B_j(x)$ are defined as the Taylor coefficients of the generating function
\[
\frac{te^{tx}}{e^t-1} = \sum_{j\geq 0} B_j(x) \frac{t^j}{j!}.
\]
So, in particular, 
\begin{equation}
\label{equ:bern1}
B_0\equiv 1.
\end{equation}
We also have
\[
\sum_{j\geq 1} \frac{dB_j}{dx}(x) \frac{t^j}{j!} = \frac{d}{dx}\frac{te^{tx}}{e^t-1} = \sum_{j\geq 1} jB_{j-1}(x) \frac{t^j}{j!}.
\]
And hence 
\begin{equation}
\label{equ:bern2}
\frac{dB_j}{dx}(x)=jB_{j-1}(x).
\end{equation}
Eqns. \eref{equ:bern1} and \eref{equ:bern2} recursively define the Bernoulli polynomials when supplemented with the relation
\begin{equation}
\label{equ:bern3}
\int_0^1 B_j(x)dx = 0
\end{equation}
for $j\geq 1$. This relation easily follows from 
\[
\int_0^1\frac{te^{tx}}{e^t-1}dx = 1.
\]

The Bernoulli numbers $B_j:=B_j(0)$ are the values of the Bernoulli polynomials at zero.

Restricting the Bernoulli polynomials to the interval $[0,1)$ and continuing $\mathbb{Z}$-periodically to $\R$, we get the Bernoulli functions $b_j:\R\rightarrow \R$. Of course $\R / \mathbb{Z} \cong S^1$, and we will use the same symbol for the functions $b_j: S^1 \rightarrow \R$.

\begin{lemma}
\label{lem:bern}
The Bernoulli functions satisfy
\[
b_j = j! (-b_1)^{*j}.
\]
where $*$ is convolution on the unit circle.
\end{lemma}
\begin{proof}
We show that the functions $j! (-b_1)^{*j}$ satisfy \eref{equ:bern1}-\eref{equ:bern3}.
The first requirement is trivially satisfied. The third is also quite obvious since
\[
\int_0^1 b_1*f=(\int_0^1 b_1) (\int_0^1 f) = 0 \cdot (\int_0^1 f) = 0.
\] 
The second requirement is fulfilled since for $j\geq 2$
\[
\frac{d(-b_1)^{*(j-1)}}{dx}= -\frac{d b_1}{dx} * (-b_1)^{*(j-1)} = (\delta-1) * (-b_1)^{*(j-1)} = * (-b_1)^{*(j-1)}
\]
where in the last equality it was used that
\[
1 * (-b_1)^{*(j-1)} = 1\cdot \int_0^1 (-b_1)^{*(j-1)} = 0.
\]
\end{proof}

The result about Bernoulli functions that is needed in Appendix \ref{app:directproof} is the following

\begin{lemma}
\label{lem:cycleweight}
\[
\int_{[0,1]^l} du_1\cdots du_l b_1(u_1-u_2)b_1(u_2-u_3)\cdots b_1(u_l-u_1)  = (-1)^l \frac{B_l}{l!}
\]
\end{lemma}
\begin{proof}
By invariance wrt. shifts of all variables we can set $u_1=0$ and multiply by the length of the unit interval, i.e., by 1. So the lhs. equals
\[
(b_1)^{*l}(0) := (b_1 * \cdots * b_1)(0) 
\]
where $*$ denotes convolution on the unit circle. But by Lemma \ref{lem:bern} and the well-known fact that $B_l=b_l(0)$, the statement is proven.
\end{proof}

\section{A supersymmetric generalization}
In \cite{markus} M. Engeli has found a supersymmetric generalization of the Hochschild cocycle $\tau_{2n}$. His cocycle is invariant under orthogonal and symplectic transformations of the super-Weyl algebra, but not invariant under orthosymplectic transformations. In this section, we generalize the above constructions to yield an (improper) $\alg{osp}$-basic representative of the \emph{cyclic} cohomology of the super-Weyl algebra. For Hochschild cohomology, we still do not know such a representative.

\subsection{Notations}
Denote the (super-)Weyl algebra by $\WA_{2n|q}$, it is
the algebra of polynomials over $\R^{2n|q}$ with Moyal product
defined by the standard symplectic structure
\[
\omega = \sum_{j=1}^n dp_j\wedge dq_j + \sum_{k=1}^q dc_k\wedge dc_k
\]
with $p_j,q_j$, $j=1,..,n$ being the even and $c_k$, $k=1,..,q$ being
the odd variables. Let
\[
\xi_\mu =
\begin{cases}
p_{(\mu+1)/2} &\quad \text{for $\mu=1,3,..,2n-1$} \\
q_{\mu/2} &\quad \text{for $\mu=2,4,..,2n$} \\
c_{\mu-2n} &\quad \text{for $\mu=2n+1,..,2n+q$}
\end{cases}
\]
and define $\omega_{\mu\nu}$ such that $\omega = \frac{1}{2}
\omega_{\mu\nu}d\xi_\mu\wedge d\xi_\nu$. Let $\epsilon(\mu)\in
\{0,1\}$ be the degree of $\xi_\mu$.
The orthosympletic Lie algebra $\alg{osp}_{2n|q}$ is the Lie
subalgebra of $\WA_{2n|q}$ given by the homogeneous quadratic
polynomials.

To write down the cochain, we first need to
introduce some notations, mostly straightforward generalizations of
those in section \ref{sec:somedefs}, with appropriate signs added.
The symbols ${}_i\pl_\mu$ and ${}_i\pr_\mu$ are defined as follows:
\begin{align*}
{}_i\pr_\mu(a_0\otimes \cdots \otimes a_k) &=
(-1)^{\epsilon(\mu)\sum_{j=0}^{i-1}|a_j|} (a_0\otimes \cdots \otimes
a_{i-1}\otimes \frac{\pr a_i}{\p \xi_\mu} \otimes a_{i+1} \cdots
\otimes a_k) \\
{}_i\pl_\mu(a_0\otimes \cdots \otimes a_k) &=
(-1)^{\epsilon(\mu)\sum_{j=0}^{i}|a_j|} (a_0\otimes \cdots \otimes
a_{i-1}\otimes \frac{\pl a_i}{\p \xi_\mu} \otimes a_{i+1} \cdots
\otimes a_k).
\end{align*}
Define further $\alpha_{rs} = \omega^{\mu\nu} \; {}_r\pl_\mu \;
{}_s\pr_\mu$.
\begin{rem}
It is easily checked that $\alpha_{rs}=-\alpha_{sr}$.
\end{rem} 
 
Let $m_k$ be the multiplication
\[
m_k(a_0\otimes \cdots \otimes a_k) = a_0\cdots a_k
\]
with the product on the right being the ``usual'' graded commutative
product of polynomials. With these notations, the Moyal product can be
written as
\[
a\star b = m_1\circ e^{\frac{1}{2}\alpha_{01}}(a\otimes b).
\]

Let $E$ be the evaluation of a polynomial at 0. Let $\Delta_k$ be the
$k$-simplex, and $\tilde{\Delta}_k$ the configuration space of $k+1$
points on the circle $\R / \Z$, with fixed ordering. We will use
coordinates $u_0,..,u_k$ on $\Delta_k$ and $\tilde{\Delta}_k$, with
the implicit understanding that they make sense on the latter space
modulo $\Z$ only.
We can embed $\Delta_k\subset \tilde{\Delta}_k$ as the subspace $\{u_0=0\}$.

We define the following differential form on $\tilde{\Delta}_k$, with
values in the normalized Hochschild cochain complex:
\[
\eta_k = E\circ m_k \circ \int d\zeta_1\cdots d\zeta_{2n+q}
\exp\left(
-\sum_{\mu,\nu=1}^{2n+q}\omega(\zeta,\zeta) +
\sum_{j=0}^k\sum_{\mu=1}^{2n+q} du_j{}_j\pr_\mu \zeta_\mu
\right)
\circ S_k
\]
where
\[
S_{k} = \prod_{0\leq i< j \leq k} e^{b_1(u_j-u_i)\alpha_{ji}}.
\]

Here $\zeta$ are coordinates on $\Pi \R^{2n|q}\cong \R^{q|2n}$
anticommuting with the forms $du_j$. The integral is over $\Pi
\R^{2n|q}$.

\begin{lemma}
\label{lem:detak}
The form $\eta_k$ satisfies
\[
d \eta_k = E\circ \pd{}{\xi_\mu} m_k \circ \int d\zeta
\omega^{\mu\nu}
\zeta_\nu
e^{
-\omega(\zeta,\zeta) +
\sum_{j=0}^k\sum_{\mu=1}^{2n+q} du_j \;{}_j\pr_\mu \zeta_\mu
}
\circ S_k
\]
\end{lemma}

\begin{proof}
\begin{align*}
d \eta_k
&=
 E\circ m_k \circ \int d\zeta
e^{ (...)
}
\circ d S_k \\
&=
-E\circ m_k \circ \int d\zeta
e^{ (...)
}
\circ \sum_{r,s=0}^k \alpha_{rs}du_s  S_k \\
&=
-E\circ m_k \circ \int d\zeta
e^{(...)
}
\circ 
\sum_{r=0}^k {}_r\pl_\mu   
\sum_{s=0}^k \omega^{\mu\nu}\; {}_s\pr_\nu du_s
   S_k \\
&=
-E\circ \pd{}{\xi_\mu} m_k \circ \int d\zeta
\omega^{\mu\nu}
e^{
-\omega(\zeta,\zeta)}
\pdl{}{\zeta_\nu}
e^{
\sum_{j=0}^k\sum_{\mu=1}^{2n+q} du_j\; {}_j\pr_\mu \zeta_\mu
}
S_k \\
&=
E\circ \pd{}{\xi_\mu} m_k \circ \int d\zeta
\omega^{\mu\nu}
\pdl{}{\zeta_\nu}
e^{
-\omega(\zeta,\zeta)}
e^{
\sum_{j=0}^k\sum_{\mu=1}^{2n+q} du_j\; {}_j\pr_\mu \zeta_\mu
}
S_k \\
&=
E\circ \pd{}{\xi_\mu} m_k \circ \int d\zeta
\zeta_\mu
e^{ (...)
}
S_k
\end{align*}
\end{proof}

\subsection{The cocycle}
We can define Hochschild cochains
\[
\tau_{2k} =  \int_{\Delta_{2k}\subset \tilde{\Delta}_{2k}} \eta_{2k}.
\]

\begin{thm}
The cochains $\tau_{2k}$ satisfy 
\[
B\tau_{2k+2} = d\tau_{2k}
\]
\end{thm}
\begin{proof}[Sketch of proof]
The Hochchild chains form a simplicial complex. Denote the $j$-th face
map by $b_j$, so that the Hochschild boundary operator on $k$-chains
is given by $\sum_{j=0}^k(-1)^jb_j$. The union of simplices $\bigcup_j
\Delta_j$ also form a simplicial space with $j$-th face map
$\sigma_j:\Delta_k\to \Delta_{k+1}$. One can check that $\eta_{k}\circ
s_j = \sigma_j^* \eta_{k+1}$. Hence we see that
\begin{align*}
d\tau_{2k} &= \sum_{j=0}^k(-1)^j\tau_{2k}\circ b_j \\
&= \sum_{j=0}^k(-1)^j \int_{\Delta_{2k}} \sigma_j^* \eta_{2k+1} \\
&= \int_{\Delta_{2k+1}} d\eta_{2k+1}
\end{align*}
by Stokes' Theorem. 

On the other hand, one can check that 
\[
B\tau_{2k+2} =  \int_{\tilde{\Delta}_{2k+1}} \eta_{2k+1}.
\]
Hence
\begin{align*}
B\tau_{2k+2} &=
\int_{\tilde{\Delta}_{2k+1}}
E\circ m_{2k+1} \circ \int d\zeta
e^{
-\omega(\zeta,\zeta) +
\sum_{j,\mu} du_j\;{}_j\pr_\mu \zeta_\mu
}
\circ S_{2k+1} \\
&=
\int_{\tilde{\Delta}_{2k+1}}
E\circ m_{2k+1} \circ \int d\zeta
e^{
-\omega(\zeta,\zeta) +
\sum_{j,\mu} d(u_j-u_0)\;{}_j\pr_\mu \zeta_\mu +du_0 \sum_{j,\mu} \;{}_j\pr_\mu \zeta_\mu
}
\circ S_{2k+1} \\
&=
\int_{\Delta_{2k+1}}
\sum_{\nu=1}^{2n+q}
E\circ \pderi{}{\xi_\nu} m_{2k+1} \circ \int d\zeta
\zeta_\nu
e^{
-\omega(\zeta,\zeta) +
\sum_{j,\mu} dv_j \; {}_j\pr_\mu \zeta_\mu 
}
\circ S_{2k+1} 
\end{align*}
In the last line we defined new coordinates $v_j=u_j-u_0$ on $\Delta_{2k+1}$ and integrated out $u_0$.
Comparing with the result of Lemma \ref{lem:detak}, the statement of the Theorem follows.
\end{proof}

\nocite{*}
\bibliographystyle{plain}
\bibliography{taucyc2} 

\end{document}